\documentclass[11pt]{article}
\usepackage[margin=1in]{geometry}                % See geometry.pdf to learn the layout options. There are lots.
\usepackage{graphicx,url}
\usepackage{parskip}
\usepackage[dvipsnames]{xcolor}
\usepackage{subcaption}
\newcounter{remarks}
\newenvironment{remarks}[1][]{\refstepcounter{remarks}\par\medskip
   \textbf{Remark~\theremarks. #1} \rmfamily}{\medskip}

\usepackage{tikz}
\usetikzlibrary{decorations.pathreplacing}                                                  % you want to use the \thanks command
\usepackage{algorithm}
\usepackage[noend]{algpseudocode}
\usepackage{mathtools,amssymb,xcolor,cite}
\usepackage{amsthm}
\allowdisplaybreaks

\newcommand{\bi}{\begin{itemize}}
\newcommand{\ei}{\end{itemize}}
\newcommand{\vo}[1]{\textcolor[HTML]{000000}{\boldsymbol{#1}}}
\newcommand{\x}{\vo{x}}
\newcommand{\f}{\vo{f}}

\newcommand{\h}{\vo{h}}
\newcommand{\vb}{\vo{v}}
\newcommand{\Vb}{\vo{V}}

\newcommand{\X}{\vo{X}}

\newcommand{\y}{\vo{y}}
\newcommand{\Mu}{\vo{\mu}}
\newcommand{\K}{\vo{K}}

\newcommand{\Q}{\vo{Q}}

\newcommand{\R}{\vo{R}}
\newcommand{\Y}{\vo{Y}}
\renewcommand{\H}{\vo{H}}
\newcommand{\I}{\vo{I}}

\newcommand{\C}{\vo{C}}
\newcommand{\A}{\vo{A}}
\newcommand{\B}{\vo{B}}

\newcommand{\F}{\vo{F}}

\newcommand{\real}{\mathbb{R}}
\newcommand{\Exp}[1]{\mathbb{E}\left[#1\right]}

\newcommand{\Sigp}{\vo{\Sigma}^{-}}

\newcommand{\Sigpp}{\vo{\Sigma}^{+}}

\newcommand{\trace}[1]{\mathbf{tr}\left(#1\right)}
\newtheorem{theorem}{Theorem}
\newtheorem{corollary}{Corollary}

\usepackage{thmtools}

\declaretheoremstyle[headfont=\normalfont]{normalhead}

\newcommand{\eqnlabel}[1]{\label{eqn:#1}}
\newcommand{\figlabel}[1]{\label{fig:#1}}
\newcommand{\eqn}[1]{(\ref{eqn:#1})}
\newcommand{\fig}[1]{fig.(\ref{fig:#1})}
\newcommand{\Fig}[1]{Fig.(\ref{fig:#1})}
\newcommand{\Rdata}{\vo{\mathcal{R}}^\text{data}_{k+1}}
\newcommand{\Rpdata}{\vo{\mathcal{R}}^\text{data}_{k+1}}
\newcommand{\Rudata}{\vo{\mathcal{R}}^\text{data}_{k+1}}

\newcommand{\Sdata}{\vo{\mathcal{S}}^\text{data}_{k+1}}
\newcommand{\Rsens}{\vo{\mathcal{R}}^\text{sensor}_{k+1}}
\newcommand{\Rmeas}{\vo{\mathcal{R}}_{k+1}}

\newcommand{\Sbar}{\bar{\vo{\mathcal{S}}}^\text{data}_{k+1}}
\newcommand{\Rbar}{\bar{\vo{\mathcal{R}}}^\text{data}_{k+1}}

\newcommand{\Stil}{\tilde{\vo{\mathcal{S}}}^\text{data}_{k+1}}
\newcommand{\Rtil}{\tilde{\vo{\mathcal{R}}}^\text{data}_{k+1}}

\usepackage{url}

\title{Privacy and Utility Aware Data Sharing for Space Situational Awareness from Ensemble and Unscented Kalman  Filtering Perspective}
\author{\begin{tabular}{cc} Niladri Das & Raktim Bhattacharya\end{tabular}\\[2mm] Aerospace Engineering, Texas A\&M University, \\ College Station, TX 77845-3141.}
\date{}                                           % Activate to display a given date or no date

\begin{document}
\maketitle
\begin{abstract}
In this paper, we present an optimization based formulation for privacy-utility tradeoff in  the Ensemble and Unscented Kalman filtering framework, with focus on space situational awareness. Privacy and utility are defined in terms of lower and upper bound on the state estimation error covariance, respectively. Synthetic sensor noise is used to satisfy these bounds, and is determined by solving an optimization problem. Given privacy and utility bounds,  we present optimization problem formulations to determine a) the maximum noise for which utility is satisfied or the estimation errors are upper-bounded, b) the minimum noise for which privacy is satisfied or the estimation errors are lower-bounded, c) the optimal noise that satisfies utility constraints and maximizes privacy, and d) the optimal noise that satisfies privacy constraints and maximizes utility. We demonstrate application of these formulations to the tracking of the International Space Station, and highlight the optimal privacy vs utility tradeoff for this dynamical system.

\end{abstract}
\section{Introduction}
As space becomes more congested, maintaining a timely and accurate picture of space activities simultaneously becomes both more important and difficult. The seriousness of this problem has been highlighted by the 2009 collision between U.S. Iridium LLC and Russian Federation Cosmos satellites, which destroyed both the satellites and added more than 2,000 additional pieces of debris in space. Prior to that, in 2007, China used a hit-to-kill interceptor to destroy their Fengyun-1C satellite. Since then, U.S. Space Surveillance Network has cataloged more than 2,200 trackable debris fragments larger than 10 centimeters originating from this collision \cite{williamsen2008satellite}. It cannot be emphasized enough that as the number of assets orbiting Earth increases  the danger and effects of collisions also increases. The new Starlink Mission of SpaceX is planned to  have 12,000 satellites \cite{foust2018spacex}. With the increased number of space objects, we can no longer assume that the space is \textit{big} and collisions between space objects will rarely occur. %Although space is big compared to the satellites, it is our inability to track all of them and the uncertainty in the data that is the major concern.

In fact, close approaches and even collisions occur with increasing regularity, with at least a dozen collisions having occurred in LEO. There are also strong indicators of at least five collisions in GEO \cite{geo1, geo2, geo3, johnson2008history, mcknight2017preliminary} since the dawn of the space age. There could easily have been more collisions in both LEO and GEO that have not been publicly disclosed \cite{oltrogge2019technical}. More recently on September 2 in 2019, the European Space Agency (ESA) maneuvered one of its Earth science satellites Aeolus, to avoid a potential catastrophic collision with a SpaceX Starlink satellite \cite{foust_2019}. The maneuver took place just about half an orbit before closest approach, indicating the importance of maintaining situational awareness in space. 

Space situational awareness (SSA) refers to the ability to view, understand and predict the physical location of natural and manmade objects in orbit around the Earth, with the objective of avoiding collisions. To address the growing SSA challenge, the U.S. and other nations, along with commercial operators, have established a new approach to exchange information regarding space objects (USSTRATCOM \& Space Data Association), hoping to increase the safety of satellite operations. While there are significant benefits to sharing of data from the SSA perspective, there are several privacy/security  related concerns from both commercial and military perspectives, which lead to conservative data-sharing policy. For example, the policy for sharing SSA data from military owned sensors, unveiled by U.S. Strategic Command in 2014, has led to removal of more SSA data from public access. This includes removal of data on the estimated size of space objects in the public satellite catalog, and limitations on what data is provided privately to satellite operators. Many of these restrictions stem from the desire to hide national utility satellites and their activities \cite{letgo}. This has led to some operators question the accuracy, and especially the completeness, of the information provided to them by the US Department of Defense (DoD). For example, some South Korean government officials estimate that their country receives data on only about 40 percent of the objects tracked by the DoD, due to sensitivity of U.S. assets \cite{lal2018global}. This has resulted in lack of confidence in the DoD-provided data, especially in the commercial sector. This lack of confidence in the data also undermines any collision warning issued based on this data, since it is expensive for operators to perform a maneuver. Thus, a security-conscious data sharing policy, impedes utility and commercial growth of the space industry.

On the other hand, improving utility raises several privacy/security concerns. National security concerns have become more critical, owing to the recent development of anti-satellite weapons and other counter-space capabilities. Accurate knowledge of space assets present risks from directed energy weapons, electronic jamming, kinetic energy threats, and other orbital threats \cite{pelton2019space, hecht2019star, harrison2018space}. From a commercial perspective, there are concerns that access to highly accurate SSA data will allow operators to assess their competitor's coverage limits, detection capabilities, and details of operation. This can be detrimental to the emerging space economy. In summary, low-accuracy SSA data increases risk of collision but reduces risk from counter-space operations and protects details of operations, i.e. it improves privacy but degrades utility. Whereas, high-accuracy SSA data improves utility but degrades privacy/security.

Currently, the US military adds synthetic noise to the public domain SSA data, much like the early GPS data model, which impacts how accurately the space objects can be tracked. Currently, this noise level is chosen conservatively and is mostly privacy or national-security conscious. With the deployment of mega-constellations in low earth orbits \cite{radtke2017interactions}, this conservative approach will not work and will impede  accurate space traffic management. Consequently, in June 2018, the Space Policy Directive 3, directed the US Department of Defense to give the publicly releasable portion of its space situational awareness data to the Commerce Department \cite{space-directive3}. This initiative will allow non-military entities to create and sell SSA services to governments and satellite operators. Thus, the SSA data can be commoditized \cite{lal2018global}, with stratified pricing models reflecting different levels of accuracy, enabling creation of several value-added services in space-data analytics. These new developments in SSA data sharing, present an important question: \textit{what should be the accuracy in the SSA data that satisfies given utility and privacy objectives?} What we lack is a framework that addresses this question in a methodological manner, enabling more informed privacy and utility preserving policies for sharing SSA data.

Privacy in dynamical systems is an emerging area of work, and has been primarily in differential privacy \cite{dwork2011differential, mcsherry2007mechanism, dwork2014algorithmic,  cortes2016differential, koufogiannis2017differential, kawano2018differential,Farokhi_2019}, where the focus is to ensure that participation of an entity or an individual does not change the outcome significantly, and also guarantees privacy of an entity based on aggregated data from many entities. Our focus here is on inferential privacy \cite{song2017composition, ghosh2016inferential, sun2017inference}, where we are trying to bound the inferences an adversary can make based on auxiliary information. In both these kinds of privacy, the mechanism for data obfuscation is either corruption of data with synthetic noise, or projection of data to a lower dimensional space. In this paper, we use synthetic noise to formulate various algorithms for privacy-utility tradeoff.

Our focus is on privacy of entities that are governed by dynamical systems. The dynamical system, which is a space object in this paper, is observed by various geographically distributed sensors. These observations are functions of states, and are noisy due to imperfect sensing. A filtering process uses these measurements to estimate the true state of the space object, which also has errors. The accuracy of the estimates is a measure of inference privacy. More error results in more privacy, and consequently less utility. For dynamical systems with Gaussian uncertainty models, Kalman filter gives the minimum variance estimates, which is function of the prior, the sensor model, and the sensor noise. In \cite{sun2017inference}, the authors manipulate the measurement data by compressing it using a linear transformation, hereby regulating the error covariance matrix or the inference privacy. This linear transformation, which creates \textit{synthetic sensors}, is designed appropriately to achieve the privacy goals. In this paper, we regulate the inference privacy by manipulating the measurement noise covariance. This is done by adding synthetic noise to the measurements. 

In this paper, we look at the trade-off between privacy and utility in the Ensemble Kalman Filtering (EnKF) \cite{evensen2003ensemble} and Unscented Kalman Filtering (UKF) framework \cite{julier1997new}, which is commonly used for data assimilation and forecasting in SSA applications. We assume that the privacy and utility constraints are specified in terms of bounds on the error variances in the state-estimates. This paper presents a convex optimization framework that determines the optimal synthetic noise to be added to the sensor data, which will satisfy the given privacy and utility bounds. The formulation, however, is quite general and is not limited to SSA related problems. To the best of our knowledge, this is the first paper that addresses the privacy-utility trade-off using synthetic noise in the EnKF/UKF framework. Our future work will be on extending this work  to other particle-filtering algorithms \cite{McCabe_2014}.

This paper is organized as follows. We first present the model for the space object used in the privacy-utility  formulation. This is followed by a very brief technical summary of Ensemble and Unscented Kalman filtering.  Key technical contributions are presented  in \S\ref{sec:main_result}, as theorems \ref{thm:utility} to \ref{thm:utility-aware}, and corollary \ref{thm:utility3}. This is followed by an example that applies the proposed privacy-utility  formulation to a realistic space-object tracking problem and considers various scenarios. The paper concludes with a summary section.

\section{Background Material}
We next present preliminaries on dynamical models of space objects, their sources of uncertainty, and measurement models. This is followed by a brief description of Ensemble and Unscented Kalman filtering. In particular, how the estimate and error variances are computed from particles, and updated from available measurements. This section also defines all the notations used to present the technical material.

\subsection{System Model}
We assume the motion of a space object is given by the differential equation (Cowell's Formulation)
\begin{align}
\ddot{\vo{r}} = \frac{\mu_{\oplus}}{r^3}\vo{r} + \vo{a}_{pert} (\vo{r},\dot{\vo{r}},t,\vo{\psi}),
\eqnlabel{dynamics}
\end{align}
where $\vo{r}$ and $\dot{\vo{r}}$ are the position and velocity of the space object in 3 dimensions respectively. The vector $\vo{a}_{\text{pert}}$ encapsulates all the perturbing accelerations of the space object other than those due to the two-body point mass gravitational acceleration. These perturbations could be due to higher-order gravity terms, atmospheric drag, solar radiations, etc; and $\vo{\psi}\in\real^d$ parameterizes these effects. The perturbation involved in the dynamics of the satellite is essentially parametric \cite{vallado2001fundamentals}, which simplifies the uncertainty propagation algorithm considerably, and particle-based methods can be used \cite{halder2011dispersion}.

Sensor observations essentially provide range, azimuth and elevation information and are derived from signals of electromagnetic radiation. Sensors receive one-way and two-way data via radar or laser measurements, which have varying degree of accuracy for a particular sensing objective such as obtaining range, range-rate and angular information for a satellite. Most sensors can provide observations at far higher frequencies than required. Thus, data density is not an issue. However, the proximity of data to other information is a more serious concern in satellite tracking. In some cases, due to mechanical design,  geographical, and political constraints, observations are limited to a small arc of the orbit, which we denote as \textit{short-arc} observations. If a satellite is observed over multiple revolutions, it is referred to as \textit{long-arc} observation. Long-arc observations are preferred because they provide more accurate determination of the satellite's orbit. Short-arc observations on the other hand are important for more accurate near-term satellite motion. Thus the sensor data is available at multiple-time scales, and a data privacy-utility formulation must account for this multi-scale nature. For space situational awareness problems, keeping track of a satellite's location in its orbit is critical, which relies on the short-arc measurements. Data from these sensors are only available when the satellites are within sensor range. Thus, predictions must be made over long intervals between these sensor updates, which is associated with larger error growths. This impacts the tradeoff between accuracy and privacy. We expect that higher accuracy data is necessary to resolve errors after long-term uncertainty propagation, and lower accuracy data is necessary to resolve errors after long-term uncertainty propagation. That is, short-arc observations can have lower accuracy than long-arc observations, to achieve a given accuracy in the state-estimate. Consequently, the data privacy-utility policy must also be time varying, accounting for the variability in the data frequency.

Let $\vo{x}\in\real^{n}$ be the state vector describing the motion of a space object. While there exists many satellite based coordinates systems to choose from for orbital state-estimation, it is well known that many of these systems are ill defined for some circular and elliptical orbits \cite{vallado2001fundamentals}. The equinoctial system eliminates this difficulty and is the best choice to quantify uncertainty in satellite dynamics. However, the data privacy-utility formulation presented in this paper is agnostic to the choice of the coordinate system. 

Let $\vo{y}\in\real^{n_y}$ be the measurement variable (measuring range, range-rate and angular information for a satellite, etc), which is mathematically defined as
\begin{align}
\vo{y} = \vo{h}(\x) + \vo{v}, \eqnlabel{sensor}
\end{align}
where $\vo{v}$ represents sensor noise, assumed to be zero-mean Gaussian random process, and $\vo{x}$ is the system state vector. 

In this paper, we develop the data privacy-utility algorithm in discrete time, and write \eqn{dynamics} and \eqn{sensor} in discrete time as
\begin{align}
  \x_{k+1} &= \f(\x_k,\vo{\psi}_k),\eqnlabel{dynmodelm}\\
  \y_k &= \h(\x_k)+\vb_k\eqnlabel{measmodelm},
\end{align}
where $\f(\cdot,\cdot)$ represents the discrete-time dynamics and $\x_k := (x,y,z,\dot{x},\dot{y},\dot{z})^T$ is the state vector, and sensor noise $\vb_k$ is assumed to be independent zero-mean Gaussian random variable, i.e. $\vb_k \sim \mathcal{N}(0,\R_k)$ and $\Exp{\vb_k\vb_l^T}=\R_k\delta_{kl}$. We further assume that $\R_k$ is a diagonal matrix, and define the inverse of $\R_k$ as the precision matrix $\vo{S}_k$. Finally, the only uncertainty considered here is the initial condition uncertainty, which is also assumed to be Gaussian i.e. $\x_0\sim \mathcal{N}(\vo{\mu}_0,\vo{\Sigma}_0)$  and independent of $\{\vb_k\}$. Other variables have the same meaning as in \eqn{dynamics} and \eqn{sensor}.

To account for the multiple-time scales in the data, and exploit the periodic nature of the satellite motion, we augment the system in \eqn{dynamics} and \eqn{sensor} to capture the motion over $q$ time steps. Starting from time $kq$, for $k\in\{1, 2, \cdots\}$, we define the augmented discrete-time system dynamics as
\begin{align}
  \X_{k+1} = \F(\X_k,\vo{\Psi}_k), \ \Y_k = \H(\X_k)+\Vb_k,\eqnlabel{augmodel}
\end{align}
where,
\begin{equation}\left.
\begin{aligned}
  \X_k &:= [\x_{kq-q+1}^T,\cdots,\x_{kq}^T]^T,\\ 
  \Y_k &:= [\y_{kq-q+1}^T,\cdots,\y_{kq}^T]^T,\\
  \vo{\Psi}_k &:= [\vo{\psi}_{kq-q+1}^T,\cdots,\vo{\psi}_{kq}^T]^T,\\ 
  \Vb_k&:=[\vb_{kq-q+1}^T,\cdots,\vb_{kq}^T]^T\sim \mathcal{N}(\vo{0},\vo{\mathcal{R}}_k),\\
    \vo{\mathcal{R}}_k &:= \textbf{diag}(\R_{kq-q+1},\cdots,\R_{kq}),
  \end{aligned}\;\;\right\}\eqnlabel{augx}\\
  \end{equation}
   denotes stacked variables. Function $\F_k(.)$ is recursively generated using $\vo{f}(.)$. The initial condition $\X_0$ is determined by propagating the uncertain initial condition $\x_0$ over $q$ time steps. It should be noted that the augmented model represents a $q$-step $q$-shift process, instead of a $q$-step sliding-window process. Augmenting the model to $q$ time steps allows $\Y_k$ to include multi-rate sensor data $\y$, by defining $q$ to be the lowest common multiple of the various sensing rates, resulting in a $q$-periodic system. In this paper, we formulate the privacy-utility policy for the multi-rate data, using the system in \eqn{augmodel}.

\subsection{Review of Ensemble and Unscented Kalman Filter}
In this paper, the data privacy-utility policy is developed in the EnKF and UKF framework, which is summarized next. The filtering process for the augmented model \eqn{augmodel} consists of uncertainty propagation using the state-dynamics to obtain the prior or model predicted estimate, and using the data from measurements to obtain the posterior estimate. In EnKF, \textit{random} samples are generated directly from the state probability density function (PDF) using standard sampling techniques. Whereas, UKF uses a \textit{deterministic} sampling technique known as the unscented transformation to pick a minimal set of sample points (called sigma points) around the mean. In both these approaches, the sample points are propagated using the nonlinear dynamics, from which the prior mean and variance of the states are computed. These are updated using measurements to arrive at the  state estimate with minimum error variance. We briefly present next, the technical details of both these approaches, which will be necessary for formulating the data-privacy vs data-utility problem in the EnKF/UKF framework. 

\subsubsection{Ensemble Kalman Filter}
In this section, we briefly present ensemble Kalman filtering. Let $\vo{\mathcal{X}}_k^{+}\in\real^{nq\times N}$ be the matrix with $N$ number of \textit{posterior} samples $\X_k^{i+}$ at time $k$, i.e.
 $$\vo{\mathcal{X}}_k^{+} := \begin{bmatrix}\X^{1+}_k & \X^{2+}_k & \cdots & \X^{N+}_k \end{bmatrix}.$$
 
The sample mean is then given by,
$$\Mu_k^{+} := \Exp{\X_k^{+}} \approx \frac{1}{N}\sum_{i=1}^{N}\X_k^{i+} =  \frac{1}{N}\vo{\mathcal{X}}_k^{+}\vo{1}_N,
$$
where $\vo{1}_N\in\real^N$ is a column vector of $N$ ones.

Defining,
$$\bar{\vo{\mathcal{X}}}_k^{+}:=\begin{bmatrix}\Mu_k^{+} & \cdots & \Mu_k^{+}\end{bmatrix} = \Mu_k^+\vo{1}^T = \frac{1}{N}\vo{\mathcal{X}}_k^+\vo{1}_N\vo{1}_N^T,$$
we can compute the variance from the samples $\vo{\Sigma}_{xx,k}^{+}$ as,
\begin{align}
& \Exp{(\X_k^{i+}-\Mu_k^{+})(\X_k^{i+}-\Mu_k^{+})^T} \approx  \vo{\mathcal{X}}_k^+\A\vo{\mathcal{X}}_{k}^{+T}\eqnlabel{sigxx}.
\end{align}
where $$\A:=\left[\frac{1}{N-1}\left(\I_N-\frac{\vo{1}_N\vo{1}_N^T}{N}\right)\left(\I_N-\frac{\vo{1}_N\vo{1}_N^T}{N}\right)\right]$$
The state of each ensemble member at the next time step is determined using the dynamics model:
  \begin{align}
    {\X}_{k+1}^{i-} &= \F({\X}_{k}^{i+},\vo{\Psi}_k^{i}).\eqnlabel{sampledyn}
  \end{align}
  
%If applied to a linear system, this ensemble approach reduces the cost associated with the time propagation of the variance matrix from $\mathcal{O}(n^3q^3)$ (classical KF) to $\mathcal{O}(n^2q^2N)$ (EnKF). 
  
In the EnKF framework, the model predicted ensembles (or prior ensembles) are corrected using measurements, as proposed by Evensen and Van Leeuwen \cite{evensen2003ensemble, evensen1996assimilation}
  \begin{align}
    {\X}_{k+1}^{i+} = {\X}_{k+1}^{i-} + \Sigp_{\vo{xy},k+1}\left(\Sigp_{\vo{yy},k+1}+\Rmeas\right)^{-1}\left(\Y_{k+1}-\vo{H}_{k+1}({\X}_{k+1}^{i-})+\boldsymbol{\epsilon}_k^i\right),
  \end{align}
where $\boldsymbol{\epsilon}_k^i$ is sampled from $ \mathcal{N}(\vo{0},\vo{\mathcal{R}}_k)$. 
Quantities $\Sigp_{\vo{xy},k+1}$ and $\Sigp_{\vo{yy},k+1}$ are computed from the samples using
\begin{align}
\Sigp_{\vo{xy},k+1} &:= \frac{1}{N-1}\left(\vo{\mathcal{X}}_{k+1}^- -\bar{\vo{\mathcal{X}}}_{k+1}^-\right)\left(\H_{k+1}(\vo{\mathcal{X}}_{k+1}^-)-\H_{k+1}(\bar{\vo{\mathcal{X}}}_{k+1}^-)\right)^T\eqnlabel{xysig}, \text{ and } \\
\Sigp_{\vo{yy},k+1} & :=  \frac{1}{N-1}\left(\H_{k+1}(\vo{\mathcal{X}}_{k+1}^-)-\H_{k+1}(\bar{\vo{\mathcal{X}}}_{k+1}^-)\right)\left(\H_{k+1}(\vo{\mathcal{X}}_{k+1}^-)-\H_{k+1}(\bar{\vo{\mathcal{X}}}_{k+1}^-)\right)^T\eqnlabel{yysig}.
\end{align}
  
%\noindent  \textbf{Remark 2:} Equation \eqn{xysig} and \eqn{yysig} allows for direct evaluation of the nonlinear measurement function $\H_k(\x)$ in calculating the Kalman gain, which is shown in \cite{Tang_2014} to hold for unbiased measurement forecasts $\{\H_k(\X_k^{i-})\}$, which we assume to be true in our work.

The variance update equation of the augmented model is given by
\begin{align}
  \Sigpp_{\vo{xx},k+1} = \Sigp_{\vo{xx},k+1} - \Sigp_{\vo{xy},k+1}\left(\Sigp_{\vo{yy},k+1}+\Rmeas\right)^{-1}{{\Sigp}_{\vo{xy},k+1}^T},\eqnlabel{sigR}
\end{align}
where $\Sigp_{\vo{xx},k+1}=\vo{\mathcal{X}}_{k+1}^-\A\vo{\mathcal{X}}_{k+1}^{-T}$. Later in the paper, we will be using \eqn{sigR} to determine the optimal privacy-utility tradeoff.

\subsubsection{Unscented Kalman Filter}
The main difference between EnKF and UKF is the generation of the samples and computation of the first two moments from the samples.
% In UKF, at time $k$, $2nq+1$ sample points (or $\sigma$ points) are generated, from the posterior mean ($\vo{\mu_k^+}$) and variance ($\vo{\Sigma}_{xx,k}^+$) \cite{julier1997new, van2001unscented}. Each of the sigma points and the associated weights are calculated using the following rule:
%\begin{align*}
%\X^{0+}_k = \vo{\mu}_k^+, \text{ and }
%\X^{i+}_k = \left\{\begin{array}{ll}
%\vo{\mu}_k^+ + \left(\sqrt{(nq+\rho)\vo{\Sigma}_{xx,k}^+}\;\right)_i, &\text{ for }i=1,\cdots,nq; \\
%\vo{\mu}_k^+ - \left(\sqrt{(nq+\rho)\vo{\Sigma}_{xx,k}^+}\;\right)_{i-nq}, &\text{ for } i=nq+1,\cdots,2nq.
%\end{array}\right.
%\end{align*}
%
%We also define,
%\begin{align*}
%\vo{\mathcal{W}}^{m} &= [\vo{\omega}_{0,k}^{(m)} \vo{\omega}_{1,k}^{(m)} \cdots \ \vo{\omega}_{2nq+1,k}^{(m)}  ]^T,\\
%\vo{\mathcal{W}}^{c} &= [\vo{\omega}_{0,k}^{(c)} \vo{\omega}_{1,k}^{(c)} \cdots \ \vo{\omega}_{2nq+1,k}^{(c)}  ]^T,
%\end{align*}
%\comment{with $\vo{\omega}_{0}^{(m)} = \rho/(nq+\rho)$, $\vo{\omega}_{0}^{(c)} = \rho/(nq+\rho) + (1-\alpha^2+\beta)$, $\vo{\omega}_{i}^{(m)} = \vo{\omega}_{i,k}^{(c)} =  1/\{2(nq+\rho)\}$} \footnote{\comment{Please explain this a bit better}}. Parameter $\rho=\alpha^2(nq+\kappa)-nq$ is the scaling parameter, and the term $\left(\sqrt{(nq+\rho)\vo{\Sigma}_{xx,k}^+}\;\right)_i$ represents the $i^\text{th}$ row of the matrix square root. In this paper, we use  $\alpha = 0.001$, $\kappa = 0$, and $\beta = 2$. 
The dynamic update step from $k\rightarrow k+1$ starts with generating deterministic points  called $\sigma$ points. Our dynamic model has no process noise.  To capture the mean $\vo{\mu}_k^{+}$ of the state vector $\X_k^{+}$, where $\X_k^{+}\in\mathbb{R}^{nq}$, as well as the error covariance $\vo{\Sigma}_{xx,k}^{+}$ 
the sigma points are chosen as 
\begin{align}
\X^{0+}_k &=  {\vo{\mu}_k^+}, \nonumber\\
\X^{i+}_k &=  {\vo{\mu}_k^+} + \Big(\sqrt{(nq+\rho) \vo{\Sigma}_{\vo{xx},k}^+}\Big)_i, i=1,...,nq, \nonumber\\
\X^{i+}_k &= { \vo{\mu}_k^+} - \Big(\sqrt{(nq+\rho) \vo{\Sigma}_{\vo{xx},k}^+}\Big)_{i-nq} ,i=nq+1,...,2nq,\nonumber
\end{align}
with associated weights as 
\begin{align}
\vo{\omega}_{0}^{(m)} &= \rho/(nq+\rho),\nonumber\\
\vo{\omega}_{0}^{(c)} &= \rho/(nq+\rho) + (1-\alpha^2+\beta),\nonumber\\
\vo{\omega}_{i}^{(m)} &=  1/\{2(nq+\rho)\}.\nonumber
\end{align}
The weight vectors are:
\begin{align}
\vo{\mathcal{W}}^{m} &= [\vo{\omega}_{0}^{(m)} \vo{\omega}_{1}^{(m)} ... \ \vo{\omega}_{2nq+1}^{(m)}  ]^T,\nonumber\\
\vo{\mathcal{W}}^{c} &= [\vo{\omega}_{0}^{(c)} \vo{\omega}_{1}^{(c)} ... \ \vo{\omega}_{2nq+1}^{(c)}  ]^T,\nonumber
\end{align}
where $\rho=\alpha^2(nq+\kappa)-nq$ is the scaling parameter, $\alpha$ is set to $0.001$, $\kappa$ is set to 0, and $\beta$ is 2 in this work. The term $\Big(\sqrt{(nq+\rho) \vo{\Sigma}_{xx,k}^+}\Big)_i$ represents $i$th row of the matrix square root. 

%The propagated state of each ensemble member at time $k+1$ is generated exactly as EnKF by using ${\X}_{k+1}^{i-} = \F({\X}_{k}^{i+},\vo{\Psi}_k^{i})$.

The propagated state of each ensemble member at time $k+1$ is generated exactly as EnKF by using \eqn{sampledyn}. However for UKF, the corresponding prior mean and variance at time $k+1$ are given by
\begin{align}
{\vo{\mu}}_{k+1}^{-} &= \vo{\mathcal{X}}_{k+1}^{-}\vo{\mathcal{W}}^{m}, \\
\vo{\Sigma}_{xx,k+1}^{-}&=\vo{\mathcal{X}}_{k+1}^{-}\B_k\vo{\mathcal{X}}_{k+1}^{-T},
\end{align}
where $\B_k:=\vo{L}\vo{L}^T$, and $\vo{L}:=\textbf{diag}(\vo{\mathcal{W}}^{c})-\vo{\mathcal{W}}^{c}\vo{1}_{2nq+1}^T$.

We next define the following terms that will be used in the following measurement update phase,
\begin{align}
\vo{\mathcal{Y}}_{k+1}^- & :=\vo{H}(\vo{\mathcal{X}}_{k+1}^-),\\ 
\bar{\vo{\mathcal{Y}}}_{k+1}^{-} & := \vo{\mathcal{Y}}_{k+1}^{-}\vo{\mathcal{W}}^{m}\vo{1}_{2nq+1}^T, \text{ and }\\
\bar{\vo{\mathcal{X}}}_{k+1}^{-}& := \vo{\mathcal{X}}_{k+1}^{-}\vo{\mathcal{W}}^{m}\vo{1}_{2nq+1}^T,
\end{align}
where  
\begin{align}
\vo{\mathcal{Y}}_{k+1}^{-} & := \begin{bmatrix}\Y^{1-}_{k+1} & \Y^{2-}_{k+1} & \cdots & \Y^{(2nq+1)-}_{k+1} \end{bmatrix}, \\
\vo{\mathcal{X}}_{k+1}^{-} &:= \begin{bmatrix}\X^{1-}_{k+1} & \X^{2-}_{k+1} & \cdots & \X^{(2nq+1)-}_{k+1} \end{bmatrix}.
\end{align}
were $\Y^{i-}_{k+1} = \vo{H}(\vo{X}_{k+1}^{i-})$ is the $i^{\mathrm{th}}$ measurement sample generated using the measurement model without the measurement noise.

For the measurement update step, we first calculate $\vo{\Sigma}_{\vo{xy},k+1}^{-}$ and $\vo{\Sigma}_{{\vo{yy}},k+1}^{-}$ as
\begin{align}\vo{\Sigma}_{\vo{xy},k+1}^{-}& :=\left(\vo{\mathcal{X}}_{k+1}^--\bar{\vo{\mathcal{X}}}_{k+1}^-\right)\times \textbf{diag}\left(\vo{\mathcal{W}}^{c}\right)\left(\vo{\mathcal{Y}}_{k+1}^{-}-\bar{\vo{\mathcal{Y}}}_{k+1}^{-}\right)^T\eqnlabel{ukf1}\\
\vo{\Sigma}_{\vo{yy},k+1}^{-}&:=\left(\vo{\mathcal{Y}}_{k+1}^{-}-\bar{\vo{\mathcal{Y}}}_{k+1}^{-}\right)\times \textbf{diag}\left(\vo{\mathcal{W}}^{c}\right)\left(\vo{\mathcal{Y}}_{k+1}^{-}-\bar{\vo{\mathcal{Y}}}_{k+1}^{-}\right)^T\eqnlabel{ukf2},
\end{align}
and then obtain the updated variance using \eqn{sigR}.

\begin{remarks}
Since the variance update equation for EnKF and UKF are identical, this allows us to formulate a common data privacy-utility policy for both the filtering frameworks. Equation \eqn{sigR} is fundamental in determining the optimal noise for privacy-utility aware data-sharing.
\end{remarks}

\section{Privacy-Utility Aware Data Sharing in EnKF and UKF}\label{sec:main_result}
As defined in \eqn{sigR}, the accuracy of the state-estimate from EnKF/UKF is quantified by $\Sigpp_{\vo{xx},k+1}$, which is defined as
$$
  \Sigpp_{\vo{xx},k+1} = \Sigp_{\vo{xx},k+1} - \Sigp_{\vo{xy},k+1}\left(\Sigp_{\vo{yy},k+1}+\Rmeas\right)^{-1}{{\Sigp}_{\vo{xy},k+1}^T},
$$
and is the minimum variance for a given $\vo{\mathcal{R}}_{k+1}$. Clearly, changing $\vo{\mathcal{R}}_{k+1}$ will change $\Sigpp_{\vo{xx},k+1}$. From an estimation perspective, $\vo{\mathcal{R}}_{k+1}$ is given by the sensing hardware, and defines the noise in the measurements $\Y_k$. However, from a data-sharing perspective, we are concerned with the privacy-utility tradeoff in sharing $\Y_k$ with end users, which can be influenced by changing  $\vo{\mathcal{R}}_{k+1}$. For a state variable, larger error covariance results in increase in privacy and decrease in its utility. On the other hand, a smaller error covariance, leads to increase in utility and decrease in privacy.

In this paper, we formulate a data privacy-utility trade-off policy in terms of synthetic noise to be added to $\Y_k$ to regulate how accurately end users are able to estimate the satellite states in the EnKF/UKF framework. Consequently, we modify \eqn{sigR}, by defining 
$$
\Rmeas := \Rsens + \Rdata,
$$
where $\Rsens$ is the known sensor noise variance and quantifies the accuracy of the measured data, and $\Rdata$ defines the additional synthetic noise to be determined that should be added to the measured data to achieve a privacy-utility tradeoff. Usually, the sensor noise is assumed to be uncorrelated, which results in a diagonal $\Rsens$ with positive entries. However, $\Rdata$ can be assumed to be a general positive semidefinite matrix. %\comment{RB: Any advantage or disadvantage in assuming full block  $\Rdata$?}

Therefore, $\Sigpp_{\vo{xx},k+1}$ is now defined as
\begin{align}
  \Sigpp_{\vo{xx},k+1} = \Sigp_{\vo{xx},k+1} - \Sigp_{\vo{xy},k+1}\left(\Sigp_{\vo{yy},k+1}+\Rsens + \Rdata\right)^{-1}{{\Sigp}_{\vo{xy},k+1}^T}. \eqnlabel{noise:privacy}
\end{align}

We generalize the formulation by defining utility in terms of variable $\vo{X}_u := \vo{M}_u\vo{X}$, 
and privacy in terms of variable $\vo{X}_p := \vo{M}_p\vo{X}$, where $\vo{M}_u$, and $\vo{M}_p$ are known matrices. Partitioning the augmented state space into privacy and utility variable is motivated by the work in \cite{Song_2018}. The authors used it to partition the state variable at a particular instant into privacy and utility variables, i.e. partitioning along the dimension of the variable. In this paper, since we augment the state-vector to include time-series data, our partitioning matrix $\vo{M}_p$ and $\vo{M}_u$ partitions the augmented state-vector in both space and time. Consequently, we can achieve privacy-utility tradeoffs in space and time, and are highlighted in the examples.

The error variance in the estimates for $\vo{X}_u$ and $\vo{X}_p$ are given by
\begin{align}
\nonumber \Sigpp_{\vo{x}_u\vo{x}_u,k+1} &:= \vo{M}_u\Sigpp_{\vo{x}\vo{x},k+1}\vo{M}^T_u,\\
& = \vo{M}_u\Sigp_{\vo{x}\vo{x},k+1}\vo{M}^T_u - \vo{M}_u\Sigp_{\vo{x}\vo{y},k+1}\left(\Sigp_{\vo{yy},k+1}+\Rsens + \Rdata\right)^{-1}{{\Sigp}_{\vo{x}\vo{y},k+1}^T}\vo{M}^T_u, \eqnlabel{xu:err} \text{ and } \\
\nonumber \Sigpp_{\vo{x}_p\vo{x}_p,k+1} &:= \vo{M}_p\Sigpp_{\vo{x}\vo{x},k+1}\vo{M}^T_p, \\
 &= \vo{M}_p\Sigp_{\vo{x}\vo{x},k+1}\vo{M}^T_p - \vo{M}_p\Sigp_{\vo{x}\vo{y},k+1}\left(\Sigp_{\vo{yy},k+1}+\Rsens + \Rdata\right)^{-1}{{\Sigp}_{\vo{x}\vo{y},k+1}^T}\vo{M}^T_p.\eqnlabel{xp:err}
\end{align}

Thus, the objective here is to determine $\Rdata$ that allows end users to estimate states $\vo{X}_u$ accurately enough, but not estimate states $\vo{X}_p$ too accurately, i.e. given  $\vo{\Sigma}_p$ and $\vo{\Sigma}_u$, we would like to determine $\Rdata$ that satisfies $\vo{\Sigma}_p \leq \Sigpp_{\vo{x}_p\vo{x}_p,k+1}$ and $\Sigpp_{\vo{x}_u\vo{x}_u,k+1} \leq \vo{\Sigma}_u$ simultaneously, where $\vo{\Sigma}_p$ defines the accuracy limit from a privacy perspective and  $\vo{\Sigma}_u$ defines the accuracy limit from a utility perspective. 

Since $\Sigpp_{\vo{x}_p\vo{x}_p,k+1}$ and $\Sigpp_{\vo{x}_u\vo{x}_u,k+1}$ depend on $\Rdata$, they are upper and lower bounded by values obtained for $\Rdata = \vo{\infty}$ and  $\Rdata = \vo{0}$ respectively, i.e.
\begin{align} 
\Sigpp_{\vo{x}_p\vo{x}_p,k+1}(0) &\le \Sigpp_{\vo{x}_p\vo{x}_p,k+1}(\Rdata) \le \Sigpp_{\vo{x}_p\vo{x}_p,k+1}(\infty) = \Sigp_{\vo{x}_p\vo{x}_p,k+1},\\
\Sigpp_{\vo{x}_u\vo{x}_u,k+1}(0) &\le \Sigpp_{\vo{x}_u\vo{x}_u,k+1}(\Rdata) \le \Sigpp_{\vo{x}_p\vo{x}_u,k+1}(\infty) = \Sigp_{\vo{x}_u\vo{x}_u,k+1},
\end{align}
where $\Sigpp_{\vo{x}_p\vo{x}_p,k+1}(\cdot)$ and $\Sigpp_{\vo{x}_u\vo{x}_u,k+1}(\cdot)$ denote the functional dependence of the variances on $\Rdata$. Consequently, these inequalities define the limits on the achievable privacy and utility. 

We next present the main results of this paper, which are convex optimization formulations for achieving optimal privacy-utility tradeoff in the general EnKF/UKF framework. We use a relaxed definition for privacy and utility, by defining them in terms of trace of the variance matrices, i.e.
\begin{subequations}
\begin{align}
\text{Privacy: } & \trace{\Sigpp_{\vo{x}_p\vo{x}_p,k+1}} \geq \gamma_p,\\
\text{Utility: } & \trace{\Sigpp_{\vo{x}_u\vo{x}_u,k+1}} \leq \gamma_u,
\end{align}
\eqnlabel{util-priv}
\end{subequations}
where $\gamma_p$ and $\gamma_u$ are user defined.

%\comment{Should $\vo{x}$ in subscript be $\vo{X}$?}

\subsection{Maximum Noise Satisfying Utility Constraints} 
Here we present a convex optimization problem, which determines the \textit{maximum} synthetic noise that can be added and still satisfy the upper bound on $\trace{\Sigpp_{\vo{x}_u\vo{x}_u,k+1}}$ and is given by the following theorem.\\

\begin{theorem}
     \label{thm:utility}
     The maximum noise that satisfies $\trace{\Sigpp_{\vo{x}_u\vo{x}_u,k+1}} \le \gamma_u $, is given by the solution of the following optimization problem
     
\begin{subequations}
\begin{align}
\min_{\Sdata \ge 0, \vo{Q}_u \ge 0} \trace{\Sdata},\eqnlabel{xu:cost}
\end{align}
subject to
\begin{align}
&\begin{bmatrix}
\vo{Q}_u - \vo{M}_u\Sigp_{\vo{x}\vo{x},k+1}\vo{M}^T_u + \vo{M}_u\Sigp_{\vo{x}\vo{y},k+1}\vo{Z}^{-1}{{\Sigp}_{\vo{x}\vo{y},k+1}^T}\vo{M}^T_u & \vo{M}_u\Sigp_{\vo{x}\vo{y},k+1} \\
{{\Sigp}_{\vo{x}\vo{y},k+1}^T}\vo{M}^T_u & \vo{Z} + \vo{Z}\Sdata\vo{Z}
\end{bmatrix} \ge 0, \eqnlabel{xu:LMI}\\ 
&\trace{\vo{Q}_u} \leq \gamma_u, \eqnlabel{xu:privacy}
\end{align}\eqnlabel{xu:opt}
\end{subequations}    
where $\vo{Z}:= \Sigp_{\vo{yy},k+1}+\Rsens$. The maximum noise in the data for which the utility constraint is satisfied is then given by $\Rdata := \left(\Sdata\right)^{-1}$.
\end{theorem}

\begin{proof}
Using \eqn{xu:err}, $\trace{\Sigpp_{\vo{x}_u\vo{x}_u,k+1}} \le \gamma_u$, is equivalent to
\begin{align*}
 \vo{Q}_u - \vo{M}_u\Sigp_{\vo{x}\vo{x},k+1}\vo{M}^T_u + \vo{M}_u\Sigp_{\vo{x}\vo{y},k+1}\left(\Sigp_{\vo{yy},k+1}+\Rsens + \Rudata\right)^{-1}{{\Sigp}_{\vo{x}\vo{y},k+1}^T}\vo{M}^T_u \ge 0,
 \end{align*}
 and $\trace{\vo{Q}_u} \leq \gamma_u$. Using matrix inversion lemma (Hua's identity), we get 
 \begin{align*}
 \left(\underbrace{\Sigp_{\vo{yy},k+1}+\Rsens}_{:=\vo{Z}} + \Rudata\right)^{-1} =
 \vo{Z}^{-1} - \left\{\vo{Z} + \vo{Z}\left(\Rudata\right)^{-1}\vo{Z}\right\}^{-1}.
 \end{align*}

Therefore, the inequality becomes 
\begin{align*}
\vo{Q}_u - \vo{M}_u\Sigp_{\vo{x}\vo{x},k+1}\vo{M}^T_u + \vo{M}_u\Sigp_{\vo{x}\vo{y},k+1}\left[\vo{Z}^{-1} - \left\{\vo{Z} + \vo{Z}\left(\Rudata\right)^{-1}\vo{Z}\right\}^{-1}\right]{{\Sigp}_{\vo{x}\vo{y},k+1}^T}\vo{M}^T_u \ge 0,
 \end{align*}
or
\begin{multline*}
 \vo{Q}_u - \vo{M}_u\Sigp_{\vo{x}\vo{x},k+1}\vo{M}^T_u + \vo{M}_u\Sigp_{\vo{x}\vo{y},k+1}\vo{Z}^{-1}{{\Sigp}_{\vo{x}\vo{y},k+1}^T}\vo{M}^T_u - \\ \vo{M}_u\Sigp_{\vo{x}\vo{y},k+1}\left\{\vo{Z} + \vo{Z}\left(\Rudata\right)^{-1}\vo{Z}\right\}^{-1}{{\Sigp}_{\vo{x}\vo{y},k+1}^T}\vo{M}^T_u \ge 0,
 \end{multline*}
 or using Schur complement we get
\begin{align*}
\begin{bmatrix}
\vo{Q}_u - \vo{M}_u\Sigp_{\vo{x}\vo{x},k+1}\vo{M}^T_u + \vo{M}_u\Sigp_{\vo{x}\vo{y},k+1}\vo{Z}^{-1}{{\Sigp}_{\vo{x}\vo{y},k+1}^T}\vo{M}^T_u & \vo{M}_u\Sigp_{\vo{x}\vo{y},k+1} \\
{{\Sigp}_{\vo{x}\vo{y},k+1}^T}\vo{M}^T_u & \vo{Z} + \vo{Z}\left(\Rudata\right)^{-1}\vo{Z}
\end{bmatrix} \ge 0.
\end{align*}

Introducing a new variable $\Sdata := \left(\Rudata\right)^{-1}$, which is the precision of the data, we get the LMI in \eqn{xu:LMI}. Therefore, maximization of $\trace{\Rudata}$ becomes minimization of  $\trace{\Sdata}$. 
\end{proof}

The above optimization is performed every time a new batch of $\Y_{k+1}$ is shared, which is corrupted using $\Rudata$.

For a special case of linear sensor model, i.e. $\H(\X_k) := \C\X_k$, we can substitute $\Sigp_{\vo{x}\vo{y}} := \Sigp_{\x\x}\C^T$ and $\Sigp_{\vo{y}\vo{y}} := \C\Sigp_{\x\x}\C^T$, in the above optimization problem.

\begin{remarks}
In the above theorem, we need to compute the inverse of $\vo{Z}:=(\Sigp_{\vo{yy},k+1}+\Rsens)$, which may be ill-conditioned, particularly when $\Sigp_{\vo{yy},k+1}$ is rank deficient and $\Rsens$ is small (corresponding to very precise sensor measurements, for example from laser ranging). This problem occurs in satellite tracking problems. We next present an alternate formulation, which does not require the matrix inverse, but the matrix square root, at the expense of increasing the problem size. This is given by the following result, assuming linear measurement model.
\end{remarks}

\begin{theorem}
\label{thm:utility2}
The maximum noise that satisfies $\trace{\Sigpp_{\vo{x}_u\vo{x}_u,k+1}} \le \gamma_u $, is given by the solution of the following optimization problem
\begin{subequations}
\begin{align}
\min_{\Sdata\ge 0, \,\vo{Q}_u \ge 0, \, \K} \trace{\Sdata},\eqnlabel{xu:cost1}
\end{align}
\text{subject to }
\begin{align}
&\begin{bmatrix}
\vo{Q}_u &\vo{M}_u(\I-\K\C)\sqrt{\Sigp_{\vo{x}\vo{x},k+1}} & \vo{M}_u\K & \vo{M}_u\K \\
\sqrt{\Sigp_{\vo{x}\vo{x},k+1}}(\I-\K\C)^T\vo{M}^T_u & \I & \vo{0} & \vo{0} \\
\K^T\vo{M}^T_u & \vo{0} & \left(\Rsens\right)^{-1} & \vo{0}\\
\K^T\vo{M}^T_u & \vo{0} & \vo{0} & \Sdata
\end{bmatrix} \ge 0, \eqnlabel{xu:LMI1}\\
&\trace{\Q_u} \le \gamma_u.
\end{align}
\end{subequations}
\end{theorem}

\begin{proof}
Recalling that the posterior variance in Kalman filtering is also given by
$
\Sigpp_{\vo{xx}} = (\I-\K\C)\Sigp_{\vo{xx}}(\I-\K\C)^T + \K\R\K^T,$
$\trace{\Sigpp_{\vo{x}_u\vo{x}_u,k+1}} \le \gamma_u$, is equivalent to
\begin{align}
& \vo{Q}_u - \vo{M}_u(\I-\K\C)\Sigp_{\vo{x}\vo{x},k+1}(\I-\K\C)^T\vo{M}^T_u + \vo{M}_u\K\left(\Rsens + \Rudata\right)\K^T\vo{M}^T_u \ge 0,
\eqnlabel{ub2}
\end{align}
along with $\trace{\vo{Q}_u} \leq \gamma_u$.

Using Schur complement we can write \eqn{ub2} as,
\begin{align*}
\begin{bmatrix}
\vo{Q}_u &\vo{M}_u(\I-\K\C)\sqrt{\Sigp_{\vo{x}\vo{x},k+1}} & \vo{M}_u\K & \vo{M}_u\K \\
\sqrt{\Sigp_{\vo{x}\vo{x},k+1}}(\I-\K\C)^T\vo{M}^T_u & \I & \vo{0} & \vo{0} \\
\K^T\vo{M}^T_u & \vo{0} & \left(\Rsens\right)^{-1} & \vo{0}\\
\K^T\vo{M}^T_u & \vo{0} & \vo{0} & \left(\Rudata\right)^{-1}
\end{bmatrix} \ge 0.
 \end{align*}

Introducing a new variable $\Sdata := \left(\Rudata\right)^{-1}$, we get the LMI in \eqn{xu:LMI1}, in terms of $\Sdata$. Maximization of $\trace{\Rudata}$ becomes minimization of  $\trace{\Sdata}$. 
\end{proof}
The above result simultaneously determines the Kalman gain $\K$ and the optimal noise in the data, for which upper bound on posterior variance is achieved.

\begin{remarks}
The discussion so far has been on adding maximum synthetic noise to existing data, for which utility is achieved. This is relevant in situations where the collected data has multiple use with different accuracy needs. Thus, it is meaningful to sense at the highest accuracy and then add synthetic noise depending on accuracy needs. However, in some applications, it may be economical to determine the sensing accuracy directly, since higher accuracy is associated with higher cost. Theorem \ref{thm:utility2} can be modified to determine the optimal sensing precision for which the prescribed accuracy in the state estimate is achieved. It is given by the following result.
\end{remarks}
\begin{corollary}
\label{thm:utility3}
The minimum sensing precision that satisfies $\trace{\Sigpp_{\vo{x}_u\vo{x}_u,k+1}} \le \gamma_u $, is given by the solution of the following optimization problem
\begin{subequations}
\begin{align}
\min_{\vo{\lambda} \ge 0,\vo{Q}_u \ge 0, \, \K} \|\vo{\lambda}\|_1,\eqnlabel{xu:cost2}
\end{align}
\text{subject to }
\begin{align}
&\begin{bmatrix}
\vo{Q}_u &\vo{M}_u(\I-\K\C)\sqrt{\Sigp_{\vo{x}\vo{x},k+1}} & \vo{M}_u\K\\
\sqrt{\Sigp_{\vo{x}\vo{x},k+1}}(\I-\K\C)^T\vo{M}^T_u & \I & \vo{0} \\
\K^T\vo{M}^T_u & \vo{0} & \vo{\mathcal{S}}^\text{sensor}\\
\end{bmatrix} \ge 0, \eqnlabel{xu:LMI2}\\
&\trace{\Q_u} \le \gamma_u. \eqnlabel{xu:trace}
\end{align}
\eqnlabel{sparse-util}
\end{subequations}
where 
$$
\vo{\mathcal{S}}^\text{sensor} := \textbf{diag}(\vo{\lambda}),
$$ 
and $\vo{\lambda}$ is the sensor precision  which is the reciprocal of the sensor noise.
\end{corollary}

\begin{proof}
The proof is similar to theorem \ref{thm:utility2}, with the sensor precision as a variable with a diagonal structure. 
\end{proof}

\begin{remarks}[Sparse Sensing:]
The $l_1$ cost in \eqn{xu:cost2} induces sparseness in the solution, i.e. many entries of the optimal $\vo{\lambda}$ are expected to be zero. These correspond to zero precision, implying that the corresponding sensor is not required to achieve the required accuracy in the state estimate, and can be eliminated from further consideration. From a system design perspective, this is quite useful. We can formulate the optimization problem with a dictionary of sensors, admitting redundancy in the sensing. The $l_1$ optimization will result in the optimal (possibly sparse) sensing precisions that will achieve the required accuracy in the state estimate.
\end{remarks}

\begin{remarks}
It is possible that the user specifies multiple partitions of the augmented state $\X$ for utility, i.e. $\X_{u_1} := \vo{M}_{u_1}\X, \dots, \X_{u_r}:= \vo{M}_{u_r}\X$, with accuracy bounds
\begin{align*}
\trace{\vo{M}_{u_1}\Sigpp\vo{M}^T_{u_1}} \leq \gamma_{u_1}, \,\cdots,\, \trace{\vo{M}_{u_r}\Sigpp\vo{M}^T_{u_r}} \leq \gamma_{u_r}.
\end{align*}

In such a case, each upper-bound constraint will add a pair of inequalities to the optimization problem, similar to \eqn{xu:LMI2} and \eqn{xu:trace}, but in terms of variable $\Q_{u_i}$. 

\end{remarks}

%\comment{We discussed the optimization problems connecting utility with the synthetic noise intensity $\Rdata$. In the succeeding section we discuss how we calculate $\Rdata$ from the privacy perspective.}

\subsection{Minimum Noise Satisfying Privacy Constraints}
We next present a convex optimization problem, which determines the \textit{minimum} synthetic noise needed to satisfy the lower bound on  $\Sigpp_{\vo{x}_p\vo{x}_p,k+1}$. It is given by the following theorem.\\

\begin{theorem}\label{thm:privacy}
The minimum noise that satisfies $\trace{\Sigpp_{\vo{x}_p\vo{x}_p,k+1}} \ge \gamma_p$, is given by the solution of the following optimization problem
\begin{subequations}
\begin{align}
\min_{\Rpdata \ge 0,\, \Q_p \ge 0} \trace{\Rpdata},\eqnlabel{xp:cost}
\end{align}
such that,
\begin{align}
&\begin{bmatrix}
\vo{M}_{p}\Sigp_{\vo{x}\vo{x},k+1}\vo{M}_{p}^{T}-\Q_p & \vo{M}_{p}\Sigp_{\vo{x}\vo{y},k+1}\\
{{\Sigp}_{\vo{x}\vo{y},k+1}^T}\vo{M}_{p}^{T} & \left(\Sigp_{\vo{yy},k+1}+\Rsens + \Rpdata\right)
\end{bmatrix}  \ge 0, \eqnlabel{xp:LMI}\\
&\trace{\Q_p} \geq \gamma_p.\eqnlabel{xp:trace}
\end{align}\eqnlabel{xp:opt}
\end{subequations} 
\end{theorem}

\begin{proof}
From \eqn{xp:err}, and $\trace{\Sigpp_{\vo{x}_p\vo{x}_p,k+1}} \ge \gamma_p$ we get the following equivalent conditions
\begin{align}
&\vo{M}_p\Sigp_{\vo{x}\vo{x},k+1}\vo{M}^T_p - \Q_p - \vo{M}_p\Sigp_{\vo{x}\vo{y},k+1}\left(\Sigp_{\vo{yy},k+1}+\Rsens + \Rpdata\right)^{-1}{{\Sigp}_{\vo{x}\vo{y},k+1}^T}\vo{M}^T_p \ge 0, \eqnlabel{xp:LMI1}\\
& \trace{\Q_p} \geq \gamma_p.
\end{align}
Using Schur complement, we can write inequality in \eqn{xp:LMI1} as the LMI in \eqn{xp:LMI} with respect to $\Rpdata$.
\end{proof}

\begin{remarks}
Like in the utility case, it is possible that the user specifies multiple partitions of the augmented state $\X$ for privacy, i.e. $\X_{p_1} := \vo{M}_{p_1}\X, \dots, \X_{p_q}:= \vo{M}_{p_q}\X$, with privacy bounds
\begin{align*}
\trace{\vo{M}_{p_1}\Sigpp\vo{M}^T_{p_1}} \geq \gamma_{p_1}, \,\dots,\, \trace{\vo{M}_{p_q}\Sigpp\vo{M}^T_{p_q}} \geq \gamma_{p_q}.
\end{align*}
In such a case, each lower-bound constraint will add a pair of inequalities to the optimization problem, similar to \eqn{xp:LMI} and \eqn{xp:trace}, but in terms of variable $\Q_{p_i}$. 
\end{remarks}

\subsection{Optimal Privacy-Utility Tradeoff}
The optimization problems in the previous two sections have addresses utility and privacy separately. In this section, we present a joint formulation for determining the optimal privacy-utility trade-off. We formulate the optimization problems around \textit{two notions of the trade-off}. 
\subsubsection{Utility-aware privacy}
The first notion is \textit{utility-aware privacy}, where the utility is specified via hard constraint $\gamma_u$ and the privacy is maximized. This results in the following optimization formulation:
This results in the following formulation:
\begin{align}
\max\; \gamma_p, \;\text{subject to } \trace{\Sigpp_{\vo{x}_p\vo{x}_p,k+1}} \geq \gamma_p, \text{ and }
\trace{\Sigpp_{\vo{x}_u\vo{x}_u,k+1}} \leq \gamma_u,
\eqnlabel{utility-aware privacy}
\end{align}
for a given $\gamma_u$. This can be generalized to multiple partitions of $\X$ for privacy and utility.
\subsubsection{Privacy-aware utility}
The second notion is \textit{privacy-aware utility}, where the privacy is specified via hard constraint $\gamma_p$ and the utility is maximized. This results in the following formulation:
\begin{align}
\min\; \gamma_u, \;\text{subject to } \trace{\Sigpp_{\vo{x}_p\vo{x}_p,k+1}} \geq \gamma_p, \text{ and }
\trace{\Sigpp_{\vo{x}_u\vo{x}_u,k+1}} \leq \gamma_u,
\eqnlabel{utility-aware privacy}
\end{align}
for a given $\gamma_p$. This can also be generalized to multiple partitions of $\X$ for both privacy and utility.

The idea is to combine the optimization formulations from theorem \ref{thm:utility} (or \ref{thm:utility2})  and theorem \ref{thm:privacy} into a single formulation. However, theorems \ref{thm:utility} (or \ref{thm:utility2}) and \ref{thm:privacy}
involve variables  $\Sdata$  and $\Rpdata$ that are constrained by $\Sdata\Rpdata = \I_{n_y}$. which is nonconvex. In this paper, we linearize this constraint about a known value of  $\Sdata$ and $\Rpdata$, and iteratively update it to arrive at a suboptimal solution. That is, initially we assume $\Sbar$ and $\Rbar$ are given and we write, 
\begin{align}
\label{eq:updateSR1}
\Sdata &:= \Sbar + \Stil \ge 0, \text{ and } \\
\label{eq:updateSR2}
\Rdata &:= \Rbar + \Rtil \ge 0,
\end{align}
and linearize $\Sdata\Rpdata = \I_{n_y}$ about $\Sbar$ and $\Rbar$, to get
\begin{align}
\Stil\Rbar+\Sbar\Rtil = 0,
\end{align}
and solve for $\Stil$ and $\Rtil$. 

We next present the complete optimization formulation for \textit{utility-aware privacy}, formulated assuming linear sensing model. It can be generalized to nonlinear sensing model by formulating it around theorems \ref{thm:utility} and \ref{thm:privacy}.

\begin{theorem}\label{thm:utility-aware}
The data noise that satisfies the utility-constraint $\trace{\Sigpp_{\vo{x}_u\vo{x}_u,k+1}} \leq \gamma_u$ for a given $\gamma_u$, and maximizes privacy, is given by the solution of the following optimization problem:
\begin{subequations}
\begin{align}
\max_{\Q_p,\,\Q_u,\,\Stil,\,\Rtil,\,\K,\, \gamma_p} \gamma_p
\eqnlabel{cost:util-aware}
\end{align}
subject to
\begin{align}
&\begin{bmatrix}
\vo{Q}_u &\vo{M}_u(\I-\K\C)\sqrt{\Sigp_{\vo{x}\vo{x},k+1}} & \vo{M}_u\K & \vo{M}_u\K \\
\sqrt{\Sigp_{\vo{x}\vo{x},k+1}}(\I-\K\C)^T\vo{M}^T_u & \I & \vo{0} & \vo{0} \\
\K^T\vo{M}^T_u & \vo{0} & \left(\Rsens\right)^{-1} & \vo{0}\\
\K^T\vo{M}^T_u & \vo{0} & \vo{0} & \Sdata
\end{bmatrix} \ge 0,\\[2mm]
&\begin{bmatrix}
\vo{M}_{p}\Sigp_{\vo{x}\vo{x},k+1}\vo{M}_{p}^{T}-\Q_p & \vo{M}_{p}\Sigp_{\vo{xx},k+1}\C^T\\
\C{\Sigp}_{\vo{xx},k+1}\vo{M}_{p}^{T} & \left(\C\Sigp_{\vo{xx},k+1}\C^T+\Rsens + \Rpdata\right)
\end{bmatrix}  \ge 0, \\[2mm]
&\trace{\Q_u} \le \gamma_u,\\
&\trace{\Q_p} \geq \gamma_p,\\
& \Sdata := \Sbar + \Stil \ge 0,\\
& \Rdata := \Rbar + \Rtil \ge 0,\\
& \Stil\Rbar+\Sbar\Rtil = 0.
\end{align}
\eqnlabel{optim:util-aware}
\end{subequations}
\end{theorem}

\begin{proof}
Builds on the proof for theorems \ref{thm:utility2} and \ref{thm:privacy}, along with linearization of the nonconvex constraint $\Sdata\Rpdata = \I_{n_y}$.\\
\end{proof}

\begin{remarks}
\label{rem:optim-privacy}
The optimization for privacy-aware utility can be formulated by replacing \eqn{cost:util-aware} with
\begin{align}
\min_{\Q_p,\,\Q_u,\,\Stil,\,\Rtil,\,\K,\, \gamma_u} \gamma_u
\eqnlabel{cost:priv-aware}
\end{align}
where $\gamma_p$ is user specified.\\
\end{remarks}

\begin{remarks}
Optimization problem \eqn{optim:util-aware} can also be generalized to multiple partitions of $\X$ for both privacy and utility, by introducing new variables  $\{\Q_{p_i}\}$, $\{\Q_{u_j}\}$, and $\{\gamma_{p_i}\}$ or $\{\gamma_{u_j}\}$.
\end{remarks}

Optimization in \eqn{optim:util-aware}, is solved repeatedly with updated values of $\Sbar$ and $\Rbar$, until there is no significant change in the cost function. For the next iteration, $\Sbar$ and $\Rbar$ are updated with the optimal $\Stil$ and $\Rtil$. The update however, is slightly different for utility-aware privacy and privacy-aware utility. For utility-aware privacy, we must ensure utility constraints are satisfied. Since the optimization in \eqn{optim:util-aware} satisfies this constraint with $\Sdata := \Sbar + \Stil$, we update $\Sbar$ and $\Rbar$ as
\begin{subequations}
\begin{align}
\Sbar &:= \Sbar + \Stil, \text{ and }\\
\Rbar &:= \left(\Sbar\right)^{-1}.
\end{align}
\eqnlabel{update:utility}
\end{subequations}

For privacy-aware utility, we must ensure privacy constraints are satisfied. Since the optimization (see remark \ref{rem:optim-privacy}) satisfies this constraint with $\Rdata := \Rbar + \Rtil$, we update $\Sbar$ and $\Rbar$ for this case, as
\begin{subequations}
\begin{align}
\Rbar &:= \Rbar + \Rtil, \text{ and }\\
\Sbar &:= \left(\Rbar\right)^{-1}.
\end{align}
\eqnlabel{update:privacy}
\end{subequations}

Pseudocode for implementing the above iterative algorithm for utility-aware privacy is presented in Algorithm \ref{alg:util-aware}. The algorithm for privacy-aware utility can be similarly developed by incorporating the changes from remark \ref{rem:optim-privacy} and \eqn{update:privacy}, in Algorithm \ref{alg:util-aware}.

\begin{algorithm}
    \caption{Algorithm for {\bf utility-aware} privacy}
    \label{alg:util-aware}
    \begin{algorithmic}
    \item Initialize: $\Sbar := \I_{n_y}$, $\Rbar := \I_{n_y}$ \texttt{\# Any other initialization may work}\vspace{1mm}
    \item Initialize: $\gamma_{p_\text{old}} := 1e^{10}$ \texttt{\# Something large}\vspace{1mm}
    \item Initialize: $\epsilon := 10^{-3}$ \texttt{\# Tolerance}\vspace{1mm}
    \item Initialize: {\tt done} := {\bf \tt false}
    \item {\bf \tt while} {\tt done} == {\bf \tt false}  
	\item \hspace{1cm}Solve optimization problem in \eqn{optim:util-aware} to get $\gamma_p$, $\Stil$ and $\Rtil$
	\item \hspace{1cm}{\ttfamily \bfseries if} $|\gamma_p - \gamma_{p_\text{old}}| \le \epsilon$\\
		   \hspace{2cm} \texttt{done} := {\bf \tt true}\\
		  \hspace{1cm} {\ttfamily \bfseries else} \\ 
		  	\hspace{2cm} $\gamma_{p_\text{old}}:= \gamma_p$\vspace{1mm}\\
		  	\hspace{2cm} Update $\Sbar$ and $\Rbar$ using \eqn{update:utility}\\
			\hspace{1cm} {\ttfamily \bfseries end}\\
	{\ttfamily \bfseries end}
    \end{algorithmic}
\end{algorithm}

\section{Numerical Simulation}
In this section, we apply the proposed algorithms for tracking the International Space Station (ISS), with its orbit defined by the following TLE set:
\begin{verbatim}
ISS (ZARYA)
1 25544U 98067A 19248.67387091 0.00001921 00000-0 41082-4 0 9997
2 25544 51.6464 322.0340 0007976 9.5374 121.4565 15.50435809187740
\end{verbatim}

The orbit is propagated in Keplerian coordinates $(a,e,i,\Omega,\omega,f)$, with uncertain initial conditions and $J4$ perturbation. 
%
%\comment{DO WE NEED THIS? :The states in the Keplerian coordinate system are defined as \cite{vallado2001fundamentals}}
%\begin{enumerate}
%\item[$a$:] semi-major axis
%\item[$e$:] eccentricity
%\item[$i$:] inclination
%\item[$\Omega$:] right ascension of the ascending node
%\item[$\omega$:] argument of perigee
%\item[$f$:] true anomaly
%\end{enumerate}
Initial condition uncertainty is assumed to be only in the semi-major axis ($a$). It is modeled as a Gaussian random variable with mean defined by the TLE set and standard-deviation equal to $1\%$ of the mean. This uncertainty is represented with $100$ samples.

Since the algorithms proposed here assume Gaussian uncertainty (EnKF and UKF), we transform the 6 dimensional orbit data from Keplerian coordinates $(a,e,i,\Omega,\omega,f)$ to Cartesian coordinates $(x,y,z,u,v,w)$, and investigate the privacy-utility tradeoff in this representation. From \fig{uncProp}, we see that $1\%$ uncertainty in $a$ causes significant increase in the state uncertainty after only one orbit, and therefore it provides a rich enough data set for investigating privacy-utility tradeoff. 

\begin{figure}[h!]
\begin{subfigure}{0.5\textwidth}
\includegraphics[width=\textwidth]{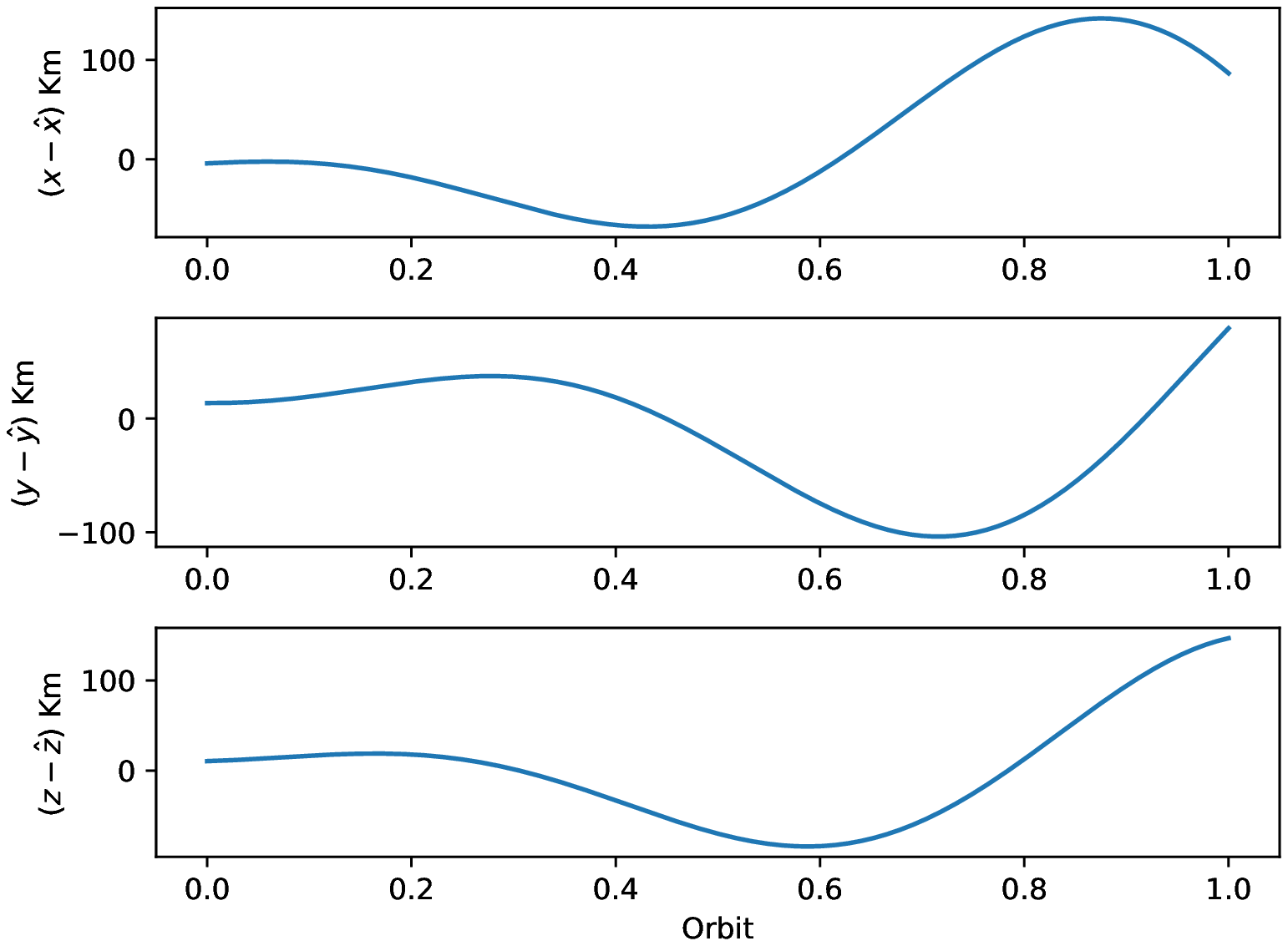}
\end{subfigure}
\begin{subfigure}{0.5\textwidth}
\includegraphics[width=\textwidth]{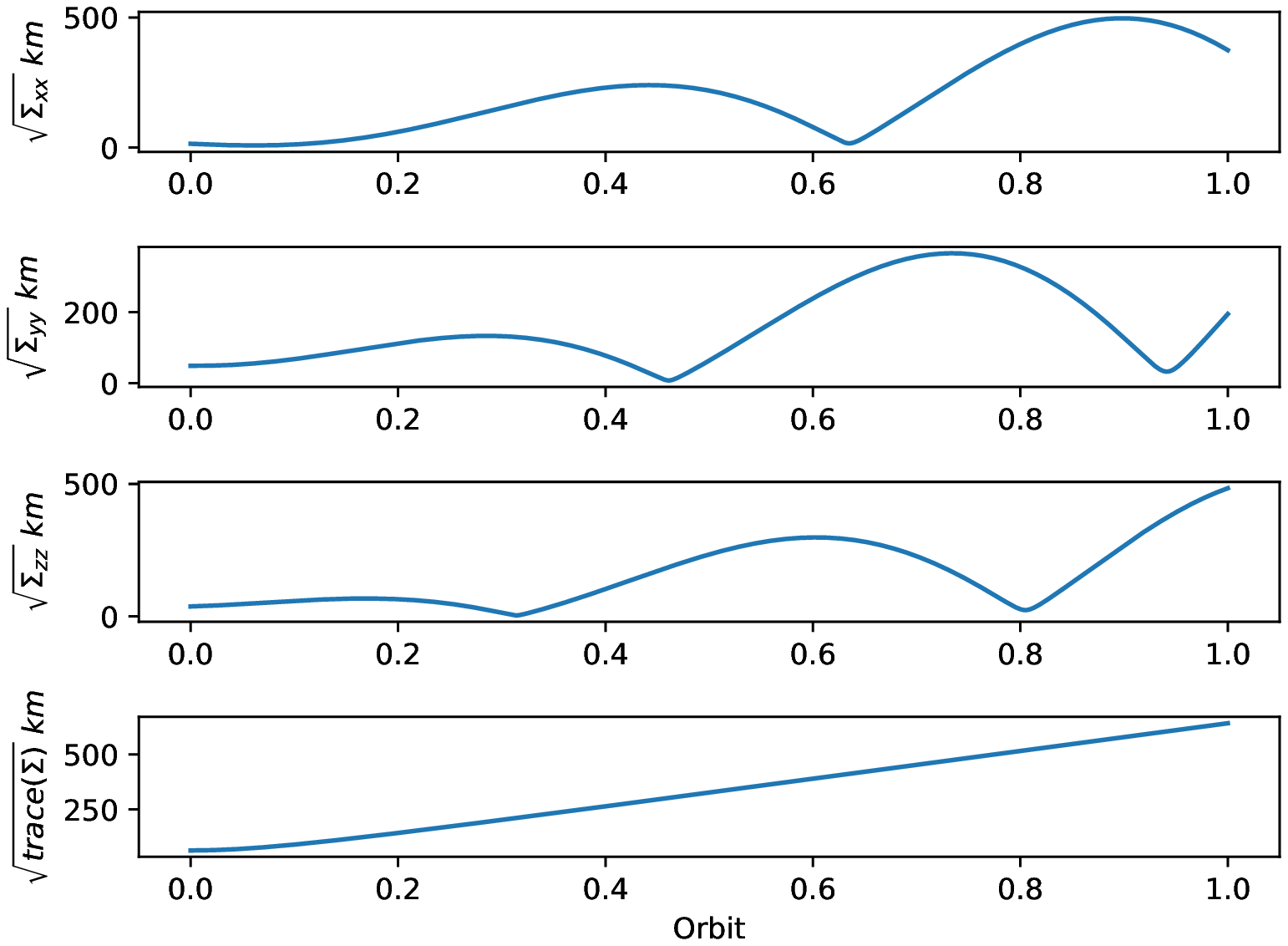}
\end{subfigure}
\caption{Growth in positional uncertainty due to $1\%$ uncertainty in the initial semi-major axis.}
\figlabel{uncProp}
\end{figure}

We also assume that we can sense $(x,y,z)$, which gives us a \textit{linear measurement model}. Consequently, we apply the corresponding algorithms to demonstrate the privacy-utility tradeoff. The orbit data is generated for $6000$ seconds, at 1 second intervals. This captures one orbit of the ISS. We extract $(x,y,z)$  at arbitrarily chosen times $0, 1600, 1900, 3400, 5100$ seconds, and treat them as measurements. These times are also expressed as $0, 0.27\,T_\text{orb}, 0.32\,T_\text{orb}, 0.57\,T_\text{orb}, 0.85\,T_\text{orb}$, where $T_\text{orb} := 6000$ seconds is the time for one orbit. The locations of these measurements, or the sensing sites, are shown in \fig{utility-orbit}. Using these five sensing sites, the objective of this example is to demonstrate:
\begin{enumerate}
\item the optimal sensor precisions that achieve the given utility in the state estimate,
\item minimum synthetic noise in the sensed data that achieves the prescribed privacy,
\item the optimal sensor precisions that achieve utility-aware privacy, i.e. maximize privacy with given utility constraints.
\end{enumerate}
These studies are performed with respect to the position of the ISS, and are presented next.

\begin{figure}[h!]
\begin{subfigure}{0.32\textwidth}
\includegraphics[width=\textwidth]{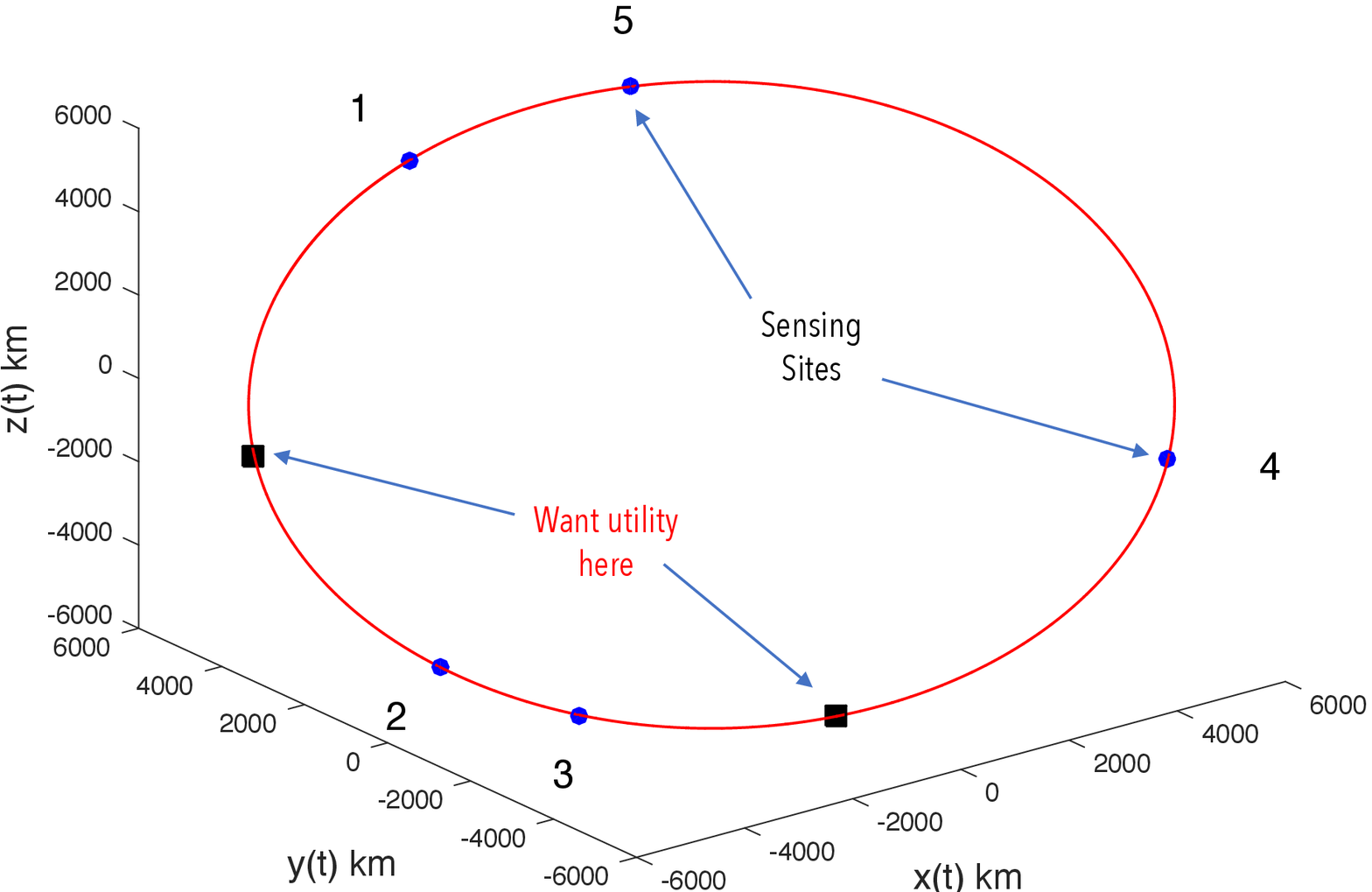}
\caption{Sensing sites and orbit locations where accurate state estimates are required.}
\figlabel{utility-orbit}
\end{subfigure}\hfill
\begin{subfigure}{0.32\textwidth}
\includegraphics[width=\textwidth]{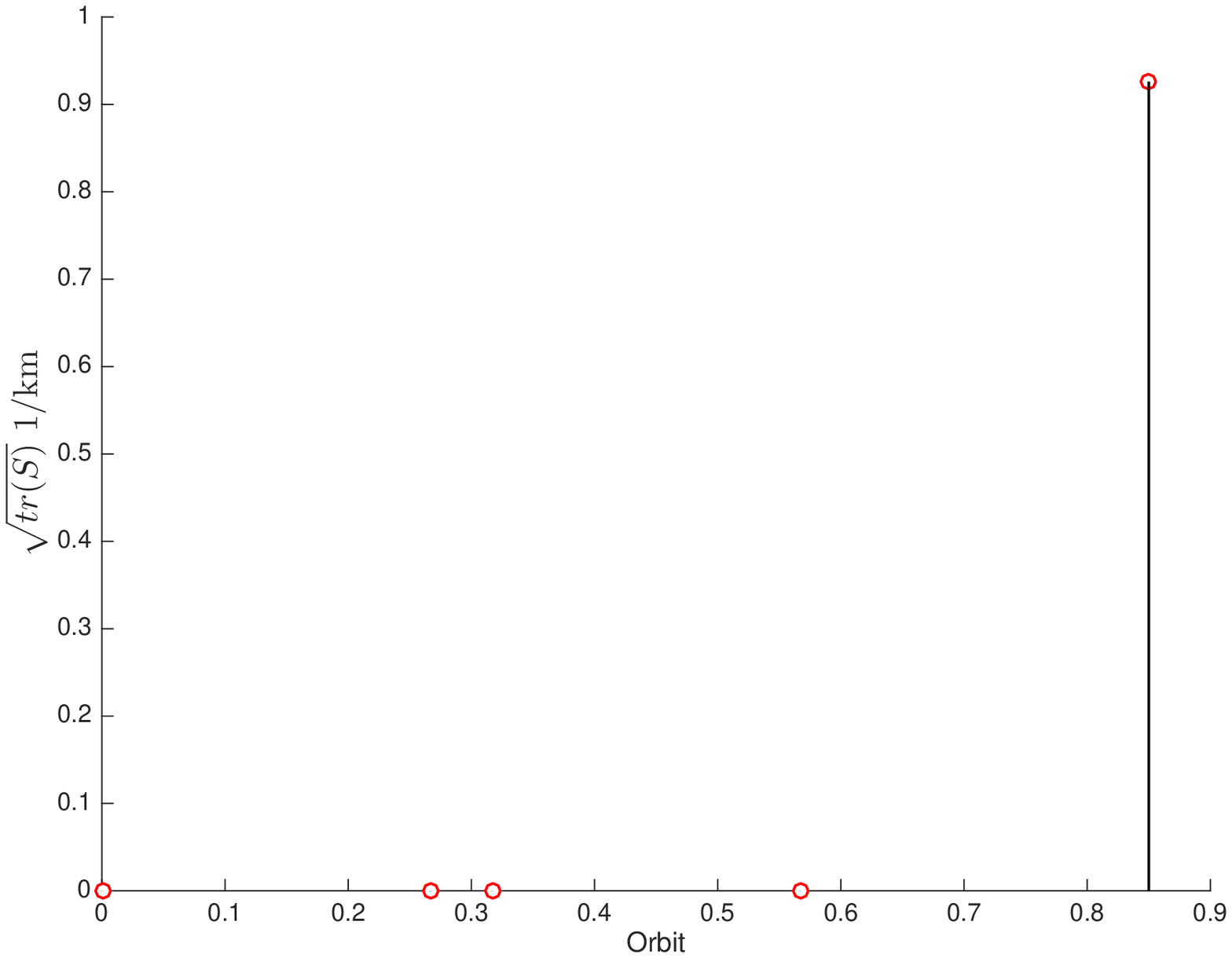}
\caption{Optimal sensor \underline{precisions}: $0, 0, 0,0,0.94$.}
\figlabel{utility-precision}
\end{subfigure}\hfill
\begin{subfigure}{0.32\textwidth}
\includegraphics[width=\textwidth]{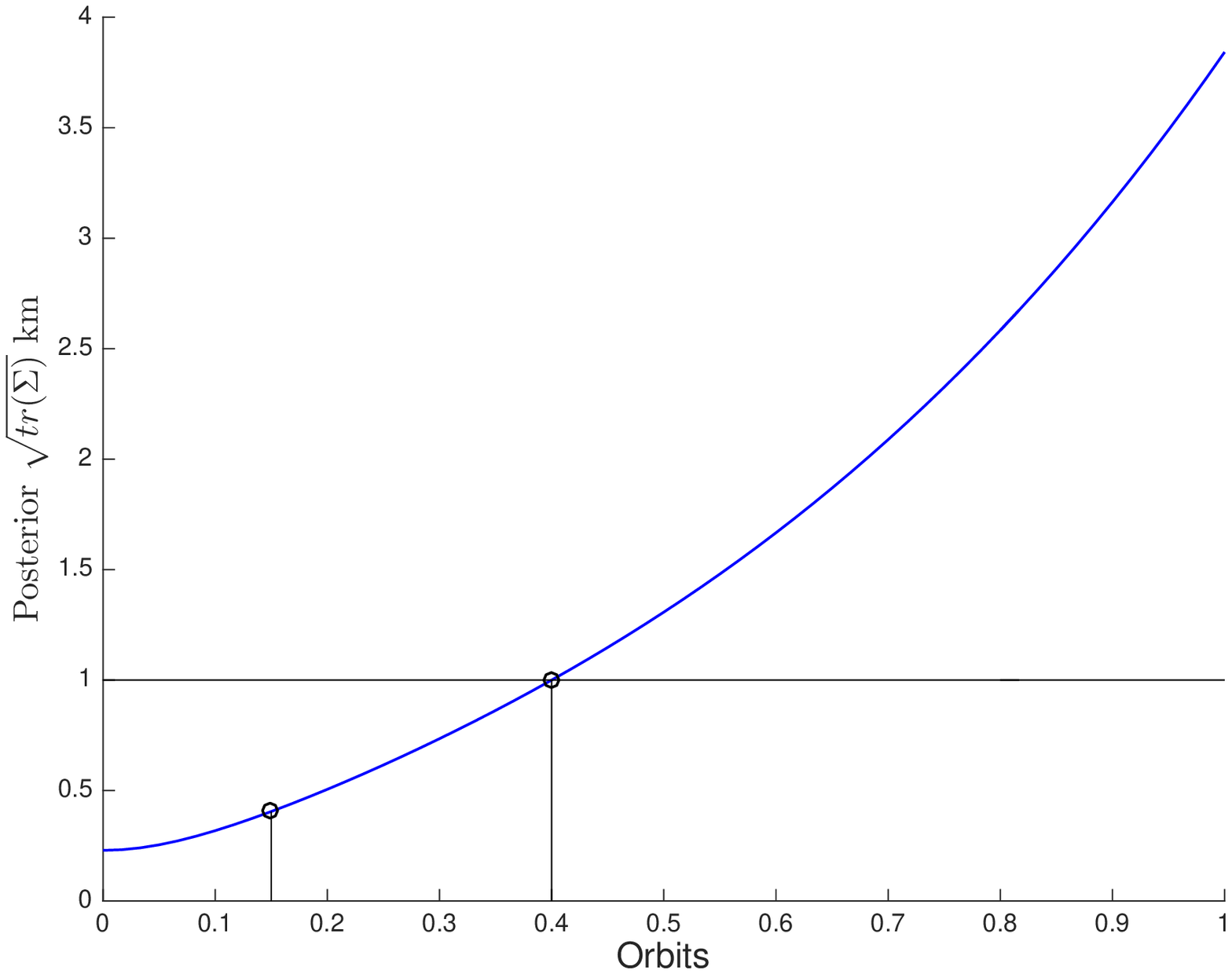}
\caption{Posterior variance satisfying the required estimation accuracy.}
\figlabel{utility-posterior}
\end{subfigure}
\caption{Optimal sensing precision satisfying utility constraints only.}
\figlabel{utility}
\end{figure}

\subsection{Minimum Precision Guaranteeing Required Utility}\label{sec:case1}
Here we determine the precisions of sensors for which the utility constraints are satisfied. This is achieved by applying the optimization problem formulation from corollary \ref{thm:utility3}.  \Fig{utility-orbit} shows the location of the sensing sites, along with the locations where the utility constraints are to be satisfied. These are imposed at times  $0.15\,T_\text{orb}$, and      $0.4\,T_\text{orb}$. Recall that utility constraints are defined by \eqn{util-priv}, and for this example we choose $\gamma_{u_i} := 1$ km, i.e. the utility constraints are $\trace{\vo{M}_{u_i}\Sigpp\vo{M}^T_{u_i}} \leq 1$. Matrices $\vo{M}_{u_i}$ are mask matrices, which extract the variance of the error in $(x,y,z)$ estimates, at times $0.15\,T_\text{orb}$ and  $0.4\,T_\text{orb}$ respectively. \Fig{utility-precision} shows the summation of the $(x,y,z)$ sensor precisions, at each of the five sensing sites, for which the utility constraints are satisfied. This is confirmed by \fig{utility-posterior}.

From \fig{utility-precision}, we see that minimization of the $l_1$ norm in \eqn{sparse-util} results in a sparse sensing architecture, i.e. many of the sensor precisions are zero. Therefore, only data from the sensing site \#5, with the indicated precision, is required to satisfy the utility constraint. There are three implications of this result. Firstly, from a sensing perspective, only sensors at site \#5 are sufficient to track the object at the specified locations with the required accuracy. This helps in sensor allocations for space object tracking. Secondly, the optimal sensing precision can be used to optimize the energy used in the radar/laser based sensing. Finally, from a data sharing perspective, if data is available from all the sensing sites -- only data from location \#5 needs to be shared. If the accuracy of the sensed data is more than required, it can be corrupted by noise defined by the reciprocal of the precision value. This will protect the economic value of the data \cite{friersoneconomic}. 

We also observe that data from the future is used to satisfy the utility constraints in the past, resulting in optimal smoothing. Since the optimization is formulated as a batch estimation, this is possible. The results are consistent with the fact that optimal smoothers generally achieve lower mean-square error than optimal filters \cite{anderson2012optimal}. This allows utility constraints in the past to be satisfied with least sensor precisions.

\subsection{Minimum Noise Guaranteeing Required Privacy}\label{sec:case2}
In this section, we apply the optimization problem from theorem \ref{thm:privacy} to determine the minimal noise in the sensor data for which the errors in the state estimates are above a prescribed value, at the specified location in the orbit. The location where privacy is required is shown in \fig{privacy-orbit}, which corresponds to time $0.82\,T_\text{orb}$. The privacy constraint is defined by \eqn{util-priv} with $\gamma_p :=  5.17^2$. This is determined by requiring $\trace{\vo{M}_{p}\Sigpp\vo{M}^T_p} \geq 10^{-4}\trace{\vo{M}_p\Sigp\vo{M}^T_p}$, where $\vo{M}_{p}$ extracts the variance of the error in $(x,y,z)$ estimates at time $0.82\,T_\text{orb}$. The sensing sites are the same as the above problem. 

\Fig{privacy-precision} shows the summation of the $(x,y,z)$ sensor noise variances, at each of the five sensing sites, for which the privacy constraints are satisfied. We see that the minimum noise, for which privacy is guaranteed, has an increasing trend, with higher noise in the vicinity of the location where privacy is required. \Fig{privacy-posterior} shows $\sqrt{\trace{\Sigpp}}$, which satisfies the privacy constraints at $t = 0.82\,T_\text{orb}$, i.e. $\sqrt{\trace{\vo{M}_{p}\Sigpp\vo{M}^T_p}} \ge 5.17$. Therefore, from a data-sharing perspective, sensor data from the 5 sites should be corrupted with synthetic noise -- defined by the optimal noise variances in \fig{privacy-precision}, to ensure the required privacy.

\begin{figure}[htb]
\begin{subfigure}{0.33\textwidth}
\includegraphics[width=\textwidth]{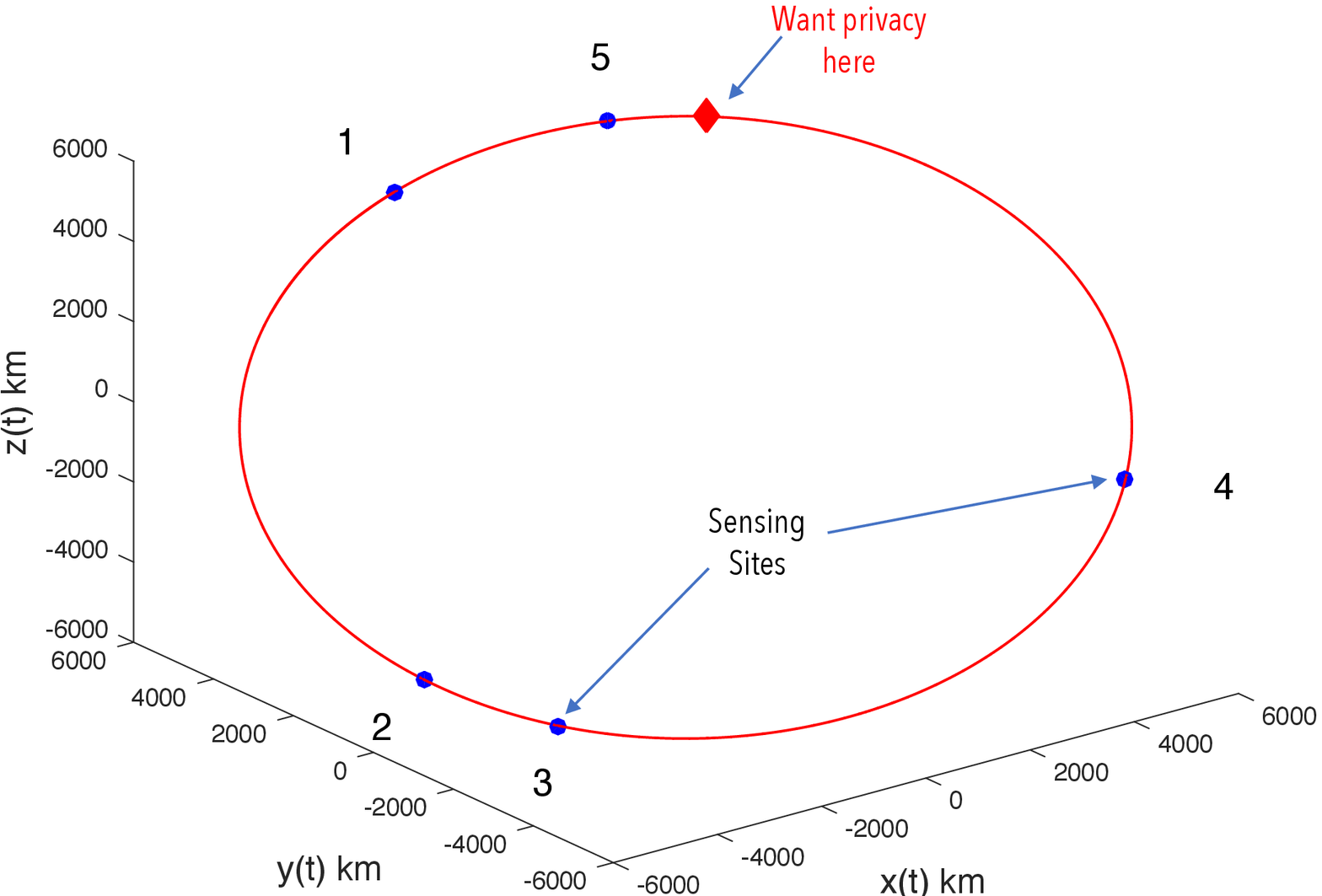}
\caption{Orbit location where privacy is required.}
\figlabel{privacy-orbit}
\end{subfigure}\hfill
\begin{subfigure}{0.33\textwidth}
\includegraphics[width=\textwidth]{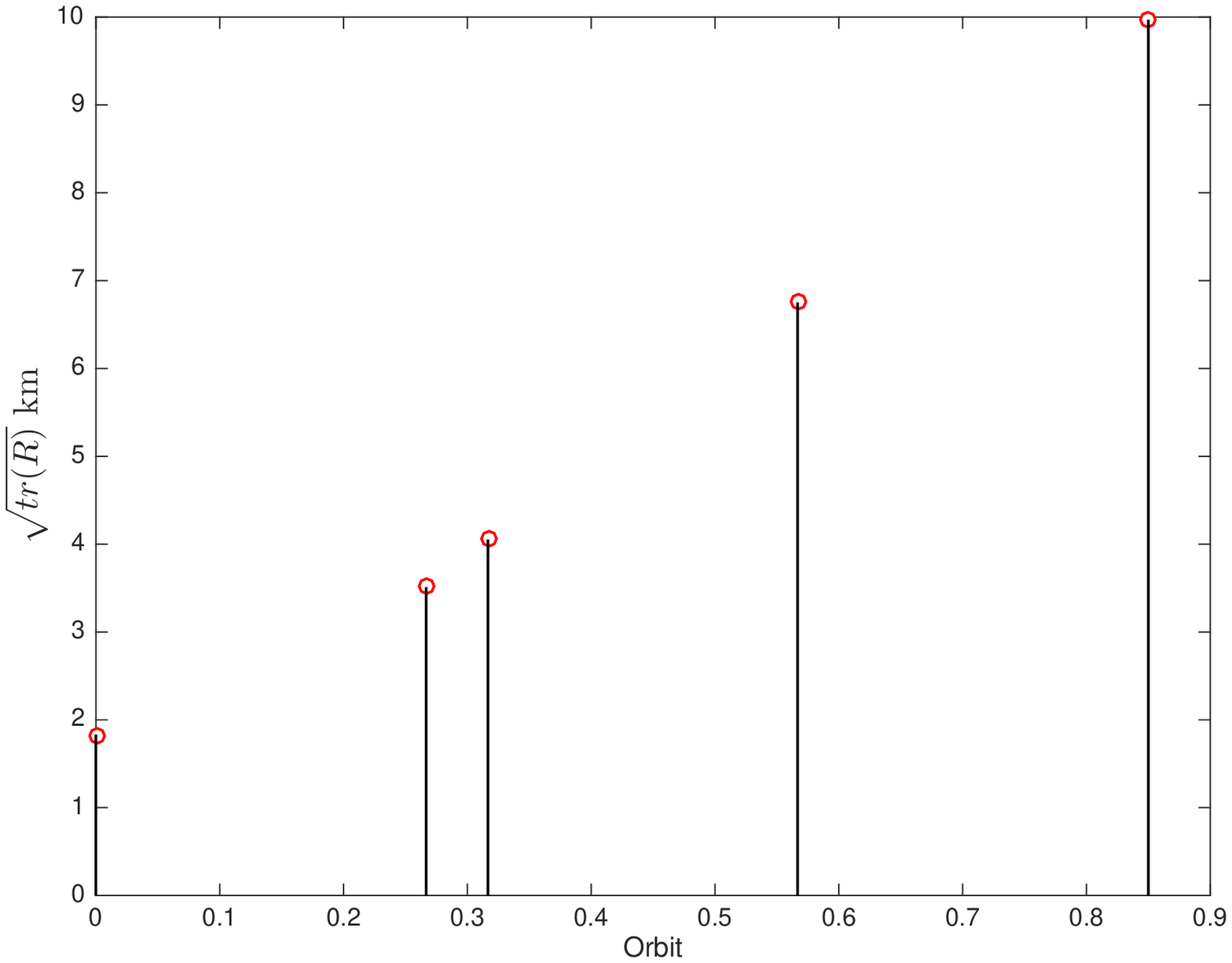}
\caption{Optimal sensor \underline{noise} with $l_1$ minimization.}
\figlabel{privacy-precision}
\end{subfigure}\hfill
\begin{subfigure}{0.33\textwidth}
\includegraphics[width=\textwidth]{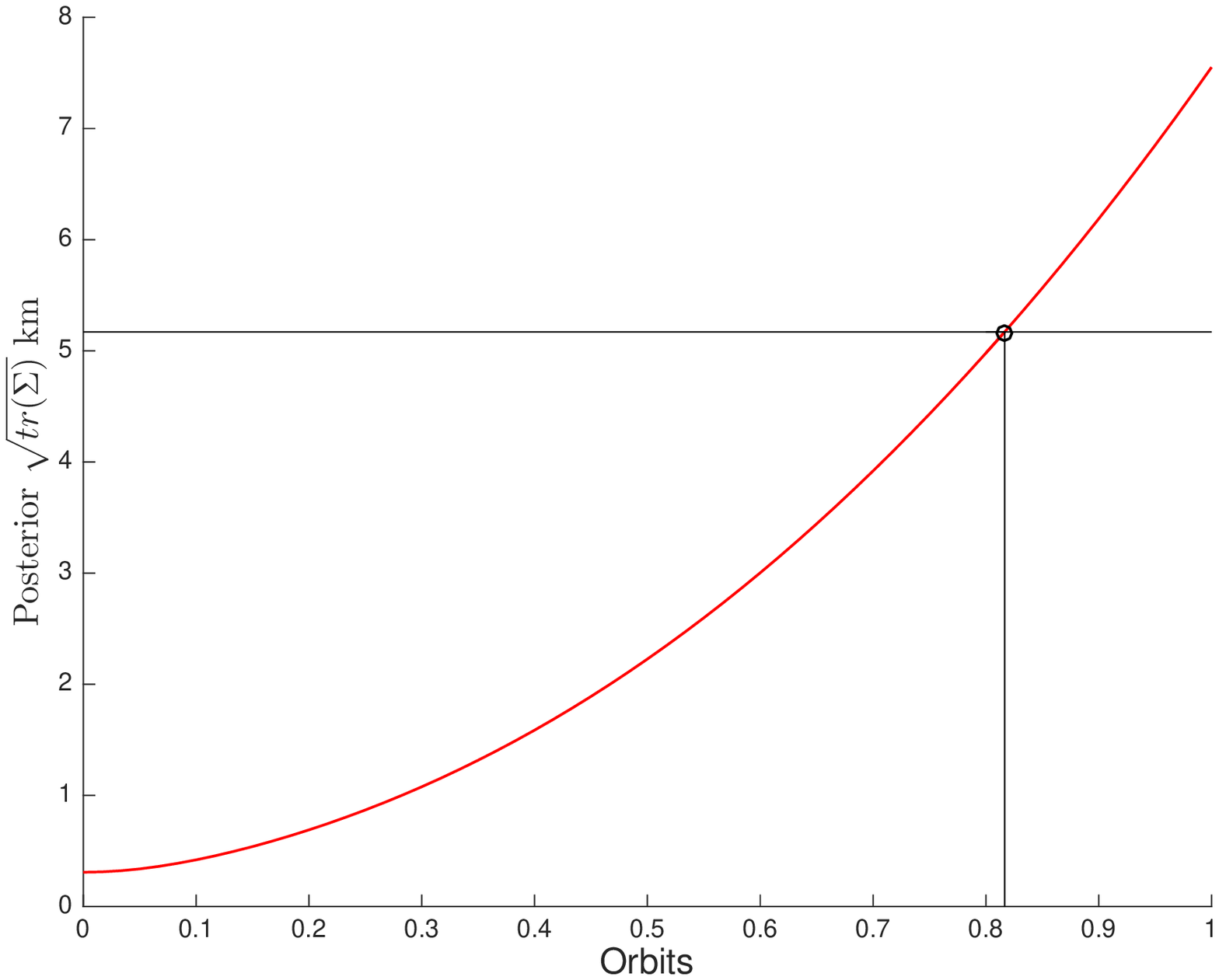}
\caption{Posterior variance satisfying the required privacy constraint.}
\figlabel{privacy-posterior}
\end{subfigure}
\caption{Optimal sensor noise for only privacy.}
\figlabel{privacy}
\end{figure}

\subsection{Utility-Aware Maximum Privacy}
We next present results obtained by applying Algorithm \ref{alg:util-aware} to the ISS data set, where privacy is maximized with utility constraints. The locations of sensors, the utility constraints, and the location where privacy is required, are the same as in the above studies. \Fig{up-precision} shows the sensor precisions for which the utility constraints (same as those in \S\ref{sec:case1}) are satisfied and privacy is maximized. The resulting posterior is shown in \fig{up-posterior}. We observe that utility constraints are satisfied at times  $0.15\,T_\text{orb}$, and $0.4\,T_\text{orb}$, while privacy is maximized at time $0.82\,T_\text{orb}$ . \Fig{up-posterior} also shows the  case when privacy is not maximized, but only utility is satisfied with minimum precision, i.e. the case discussed in \S\ref{sec:case1}. We see that utility-aware maximum privacy is able to achieve about $1.63$ times more privacy. The value of $\sqrt{\trace{\vo{M}_{p}\Sigpp\vo{M}^T_p}}$ in  \S\ref{sec:case1} is $2.26$, whereas $\sqrt{\trace{\vo{M}_{p}\Sigpp\vo{M}^T_p}}$ in this case is $4.35$. Thus, we can see the benefit Algorithm \ref{alg:util-aware}. %However, in comparison with \fig{privacy-posterior}, we see how the utility constraints limit the achievable privacy.

\begin{figure}[htb]
\begin{subfigure}{0.32\textwidth}
\includegraphics[width=\textwidth]{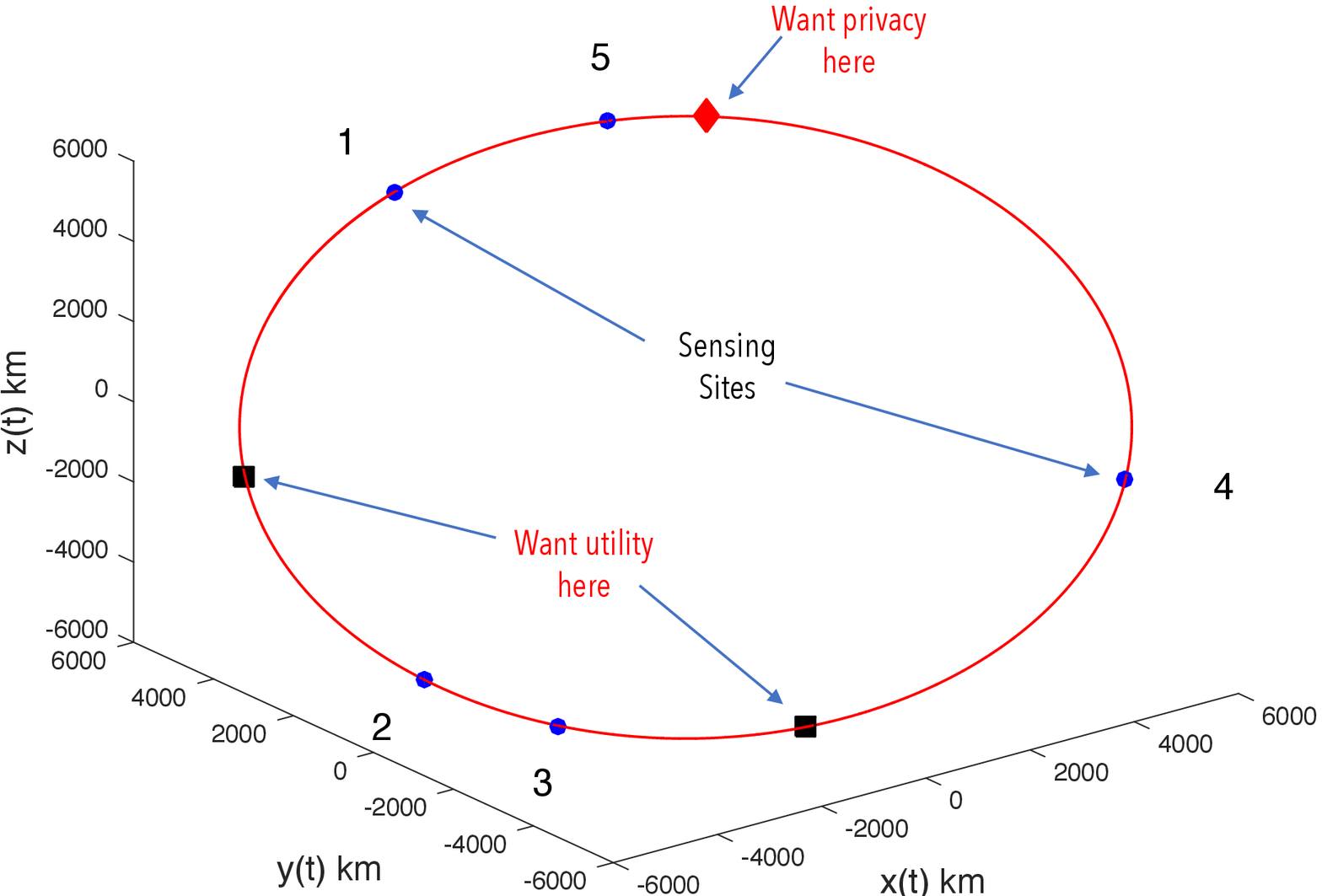}
\caption{Orbit locations where utility is constrained, and location where privacy is maximized.}
\figlabel{up-orbit}
\end{subfigure}\hfill
\begin{subfigure}{0.32\textwidth}
\includegraphics[width=\textwidth]{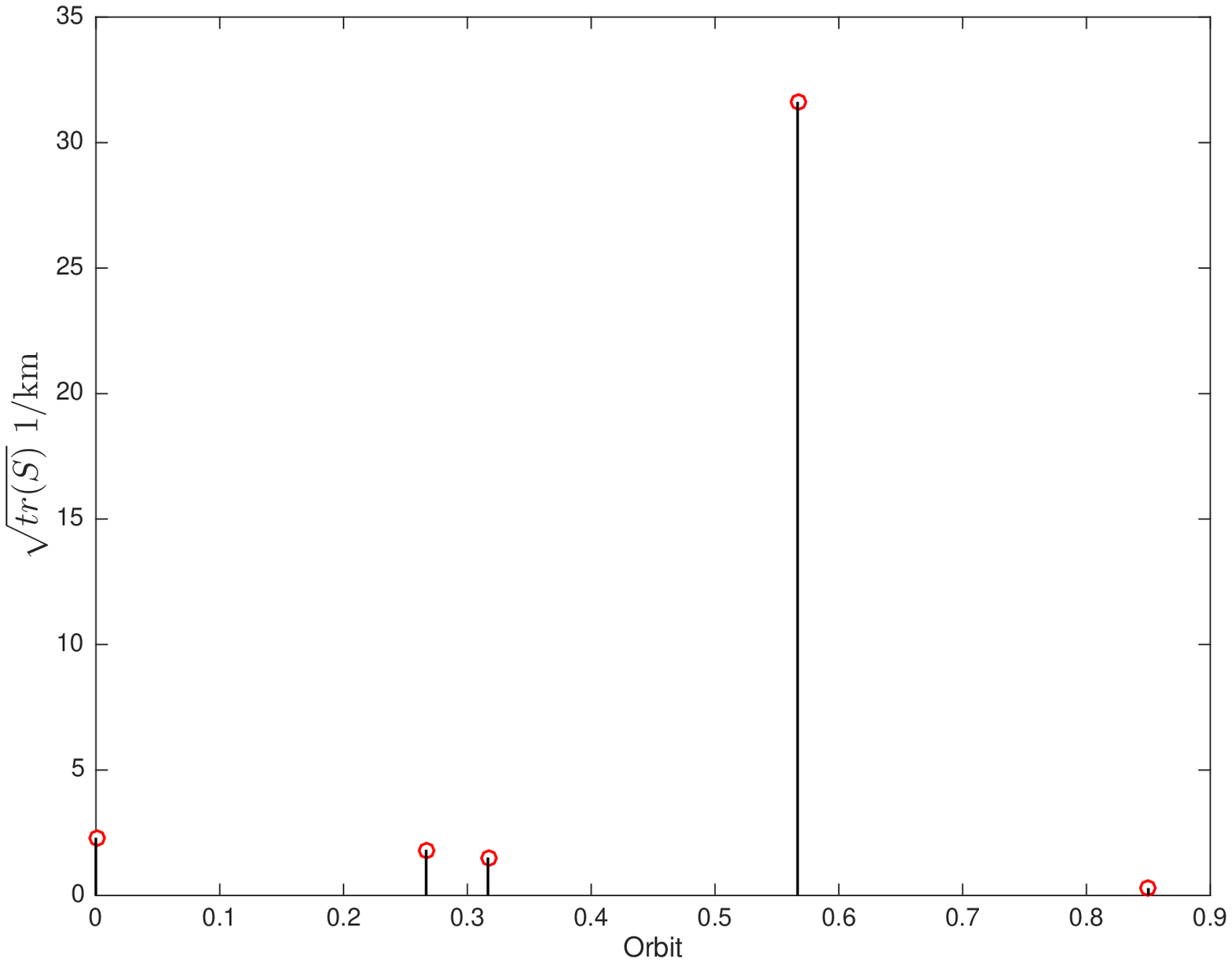}
\caption{Sensor \underline{precisions} for maximum utility-aware privacy: $2.30, 1.82, 1.51, 31.62, 0.29$.}
\figlabel{up-precision}
\end{subfigure}\hfill
\begin{subfigure}{0.32\textwidth}
\includegraphics[width=\textwidth]{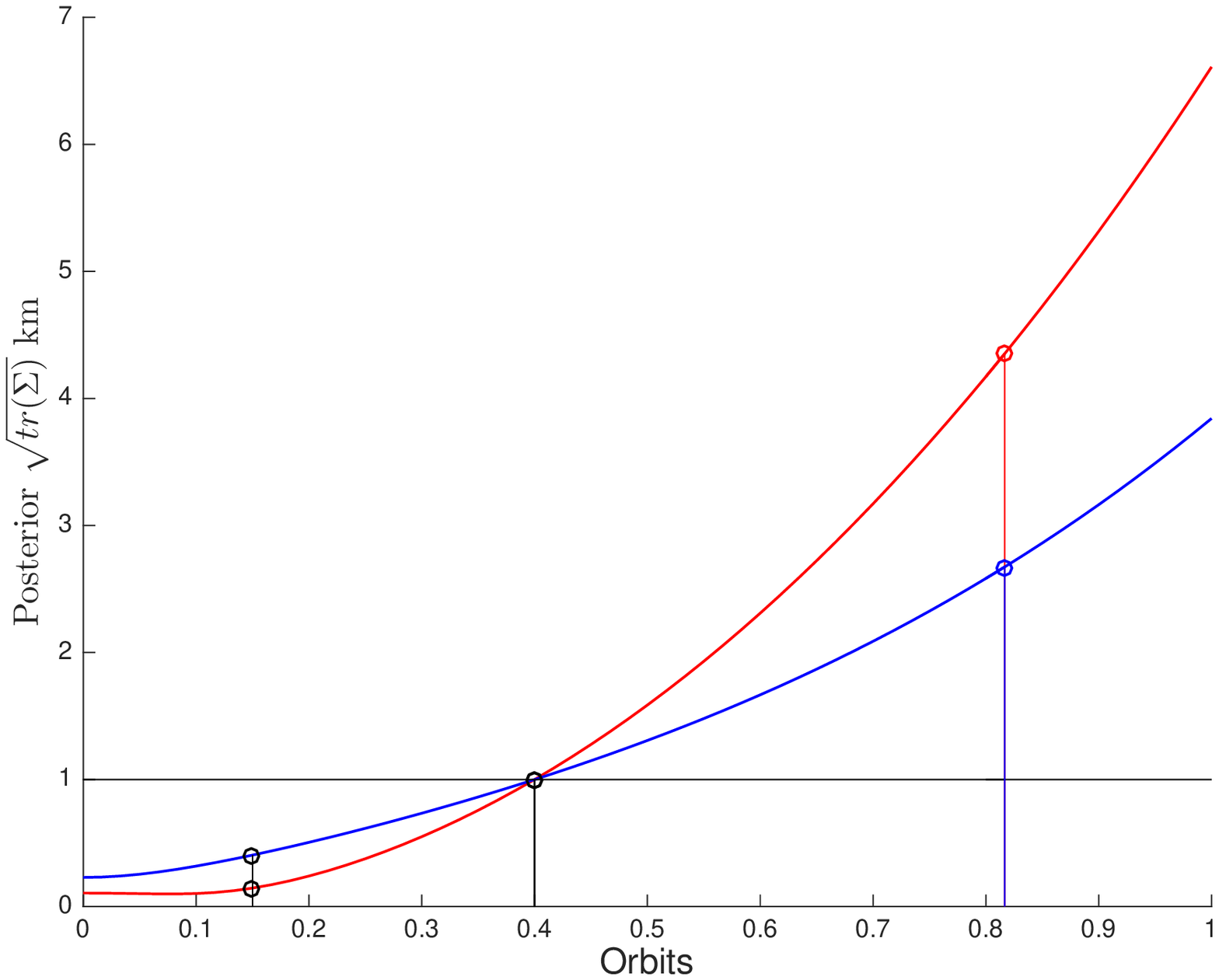}
\caption{Red: Posterior variance for maximum utility-aware privacy. Blue: \fig{utility-posterior} shown again.}
\figlabel{up-posterior}
\end{subfigure}
\caption{Optimal sensor precision for utility-aware privacy over one orbit of the ISS.}
\figlabel{up}
\end{figure}

\Fig{up-precision} shows that the utility constraints are satisfied using most precise data from the future (i.e. site \#4) , which implies that the optimization primarily applies smoothing to satisfy the utility constraints. This is similar to the results from \fig{utility-precision}, where smoothing was used to satisfy the utility constraints. 

We also see that privacy is maximized by reducing the precision at site \#5. Data from site \#5 is from the future, and incorporating this data would reduce uncertainty at all past locations due to smoothing. This would not maximize privacy. Low precision data from site \#5 essentially means that no update is made to the prior beyond the data from the sensing site \#4, and this maximizes the privacy. Consequently, data from site \#4 plays an important role. It is used to satisfy the uncertainty constraints at the utility sites using smoothing, and maximize uncertainty at the privacy site using prediction. 

\begin{remarks}
The examples above, are shown with data from one orbit, where privacy is maximized after the utility. This is a simpler scenario, because state uncertainty grows without update and hence it is easier to achieve higher uncertainty, and hence privacy, at later times.
\end{remarks}

In the next example, we consider an interesting scenario where privacy is maximized in between two locations where utility is constrained. For this scenario, we consider data from \textit{five} orbits of the ISS, with data saved every 100 seconds and measurements available at arbitrarily chosen  times $0$, $0.15\,T_\text{orb}$, $0.82\,T_\text{orb}$, $1.65\,T_\text{orb}$, $3.32\,T_\text{orb}$, $4.15\,T_\text{orb}$, and $4.98\,T_\text{orb}$. Therefore, there are $7$ sensing sites, measuring $(x,y,z)$. The utility constraints are imposed at times $0.48\, T_\text{orb}$ and $4.82\, T_\text{orb}$, with $\trace{\vo{M}_{u_i}\Sigpp\vo{M}^T_{u_i}} \leq 1$. Privacy is maximized at the time $2.48\,T_\text{orb}$. Therefore, privacy is maximized in between the two times where utility is constrained. The privacy and utility time points are chosen arbitrarily.

\begin{figure}[htb]
\begin{subfigure}{0.32\textwidth}
\includegraphics[width=\textwidth]{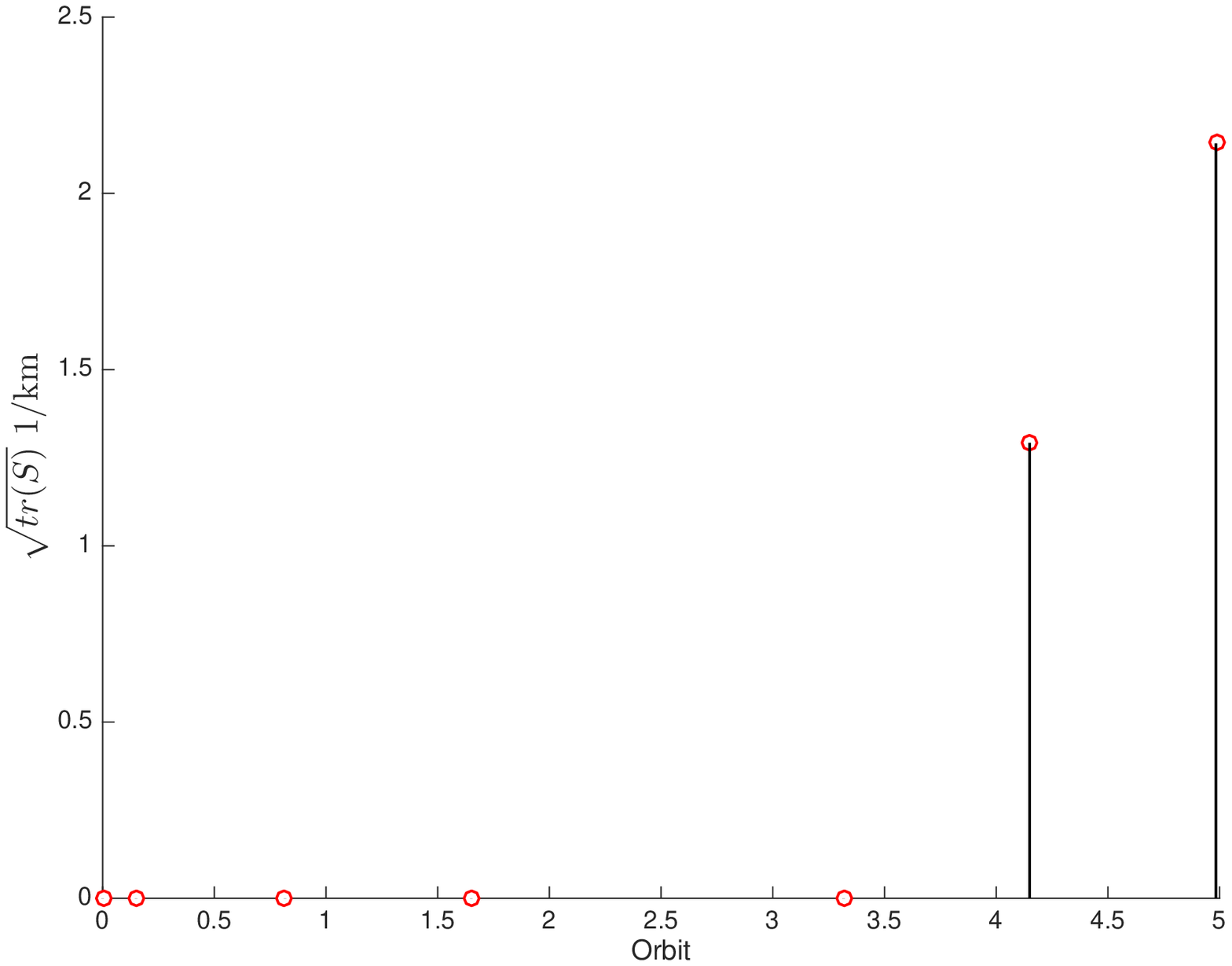}
\caption{Minimum precision satisfying only utility constraints: $ 0, 0, 0, 0, 0, 1.29, 2.14$.}
\figlabel{u-precision2}
\end{subfigure}\hfill
\begin{subfigure}{0.32\textwidth}
\includegraphics[width=\textwidth]{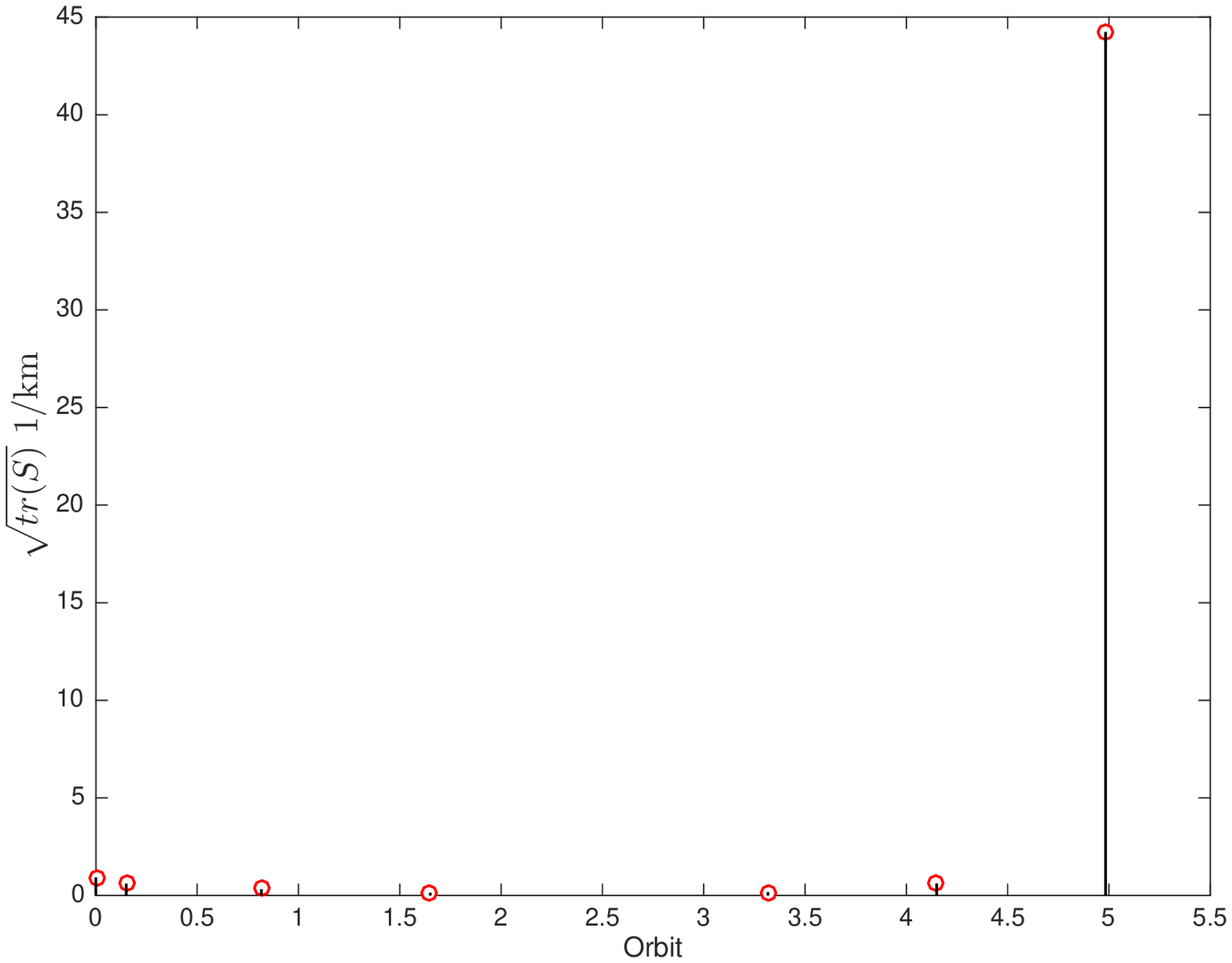}
\caption{Sensor precisions for maximum utility-aware privacy: $0.93, 0.61, 0.33, 0.13, 0.17, 0.6, 44.23$.}
\figlabel{up-precision2}
\end{subfigure}\hfill
\begin{subfigure}{0.32\textwidth}
\includegraphics[width=\textwidth]{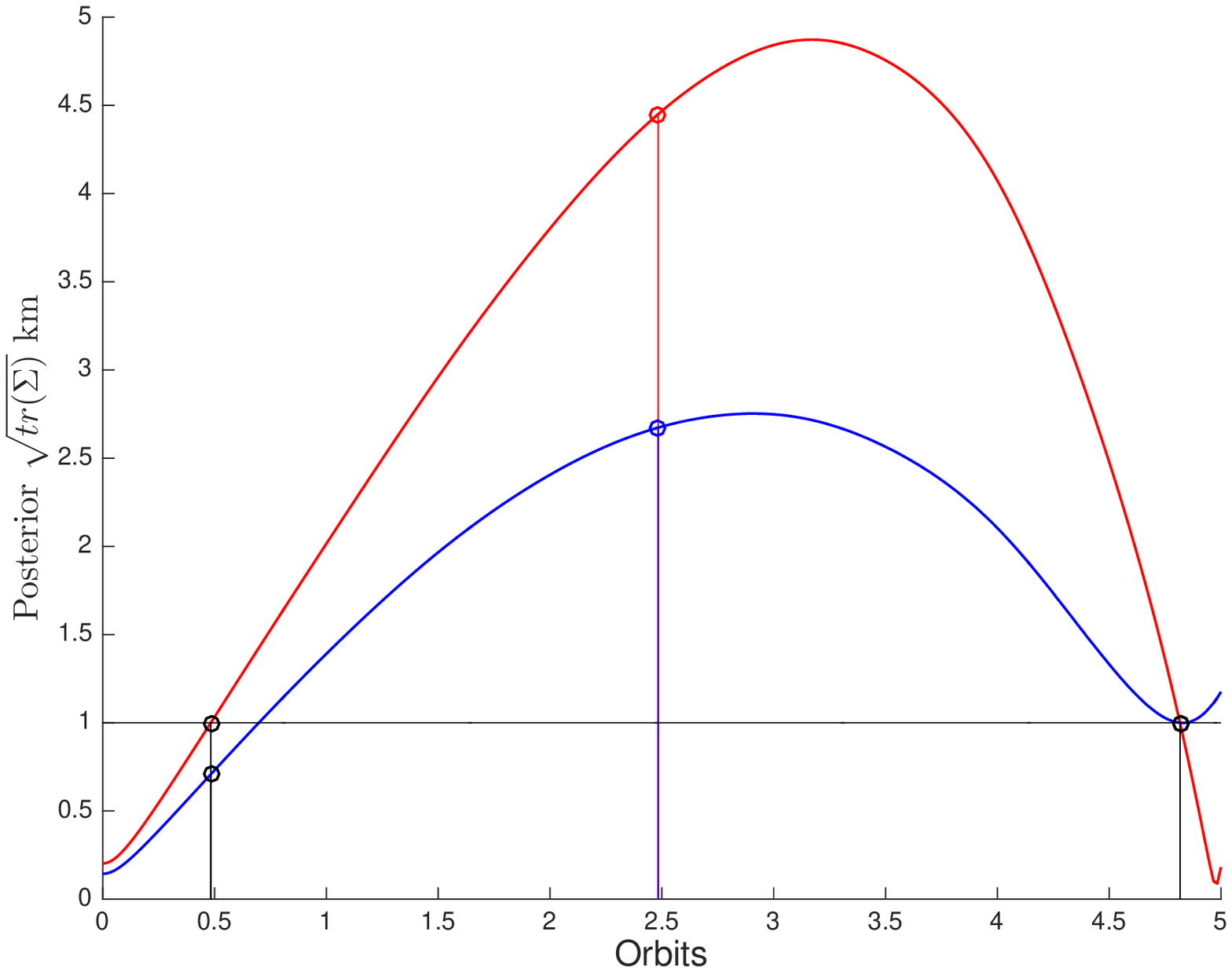}
\caption{Red: Posterior variance for maximum utility-aware privacy. Blue: Posterior variance satisfying utility constraints only.}
\figlabel{up-posterior2}
\end{subfigure}
\caption{Optimal sensor precision for utility-aware privacy over \underline{five} orbits of the ISS.}
\figlabel{up2}
\end{figure}

\Fig{up2} shows the results from the utility-aware privacy optimization for this case. In \fig{u-precision2}, the optimal sensor precisions that  satisfy the utility constraints are shown. We observe that data from sites \#1 to \#5 are not required to satisfy the utility constraints. Only data from sites \#6 and  \#7 are required, with the given precisions. The corresponding $\sqrt{\trace{\Sigpp}}$ is shown in \fig{up-posterior2} in blue, and we can see that the utility constraints are satisfied at times $0.48\, T_\text{orb}$ and $4.82\, T_\text{orb}$. Achieving the utility with minimum precision implicitly achieves a certain level of privacy at time $2.48\,T$, as we can see $\sqrt{\trace{\Sigpp}}$ increases between times $0.48\,T$ and $4.82\,T$. \Fig{up-posterior2} also shows in red, $\sqrt{\trace{\Sigpp}}$  from utility-aware privacy maximization. We can see that it achieves significantly higher privacy at time $2.48\,T$. The value of $\sqrt{\trace{\vo{M}_{p}\Sigpp\vo{M}^T_p}}$  from utility-only optimization is $2.67$, and the value of $\sqrt{\trace{\vo{M}_{p}\Sigpp\vo{M}^T_p}}$ from utility-aware privacy maximization is $4.45$, resulting in an improvement by a factor of $1.67$. The corresponding sensor precisions are shown in \fig{up-precision2}. Therefore, for this example too, we see that Algorithm \ref{alg:util-aware} is able to achieve higher privacy, while satisfying the utility constraints. 

We next discuss the convergence property of  Algorithm \ref{alg:util-aware}. For the results shown in \fig{up} and \fig{up2}, the respective convergences are shown in \fig{algo-convergence}, where $\Delta J:=|\gamma_p - \gamma_{p_\text{old}}|$.
\begin{figure}[htb]
\begin{subfigure}{0.45\textwidth}
\includegraphics[width=\textwidth]{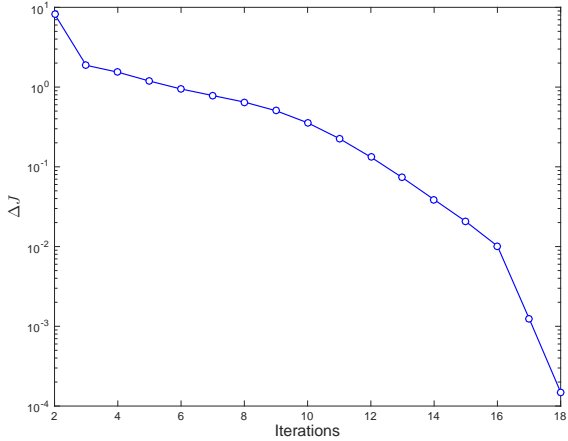}
\caption{Convergence of Algorithm \ref{alg:util-aware} for the 1-Orbit problem shown in \fig{up}.}
\figlabel{algo-conv1}
\end{subfigure}\hfill
\begin{subfigure}{0.45\textwidth}
\includegraphics[width=\textwidth]{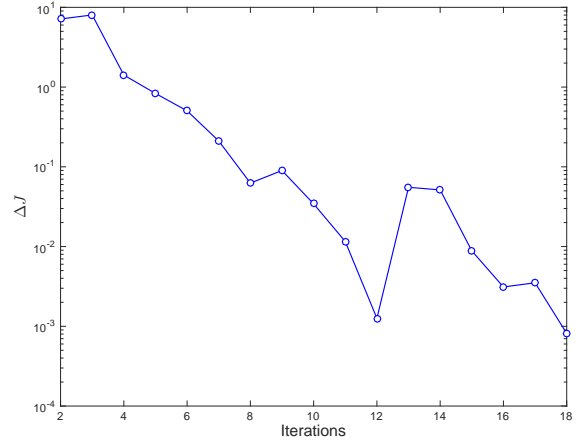}
\caption{Convergence of Algorithm \ref{alg:util-aware} for the 5-Orbit problem shown in \fig{up2}.}\figlabel{algo-conv2}
\end{subfigure}\hfill
\caption{Convergence of Algorithm \ref{alg:util-aware}.}
\figlabel{algo-convergence}
\end{figure}

Empirically, we observe that in both the cases the algorithm convergences to the optimal solution within $18$ iterations, for $\epsilon := 10^{-3}$. The associated computational times are summarized in table \ref{tab:cpu}. For every iteration, the optimization problem for the $1$-orbit problem takes an average of $3.5$ seconds. Whereas, for the $5$-orbit problem, optimizations take an average of $235.93$ seconds, which is higher due to larger problem size. The computational times for the utility-only optimization are also shown in the table. We observe that for $1$ and $5$ orbit problems, the computational times are $0.5$ and $28.26$ seconds respectively. These times are obtained in MATLAB, with YALMIP \cite{lofberg2004yalmip} as the parser and MOSEK \cite{andersen2000mosek} as the solver, executing in a Mac Book Pro with 2.3 GHz Quad-Core Intel Core i7 processor and 16 GB (1600 MHz DDR3) memory.

\begin{table}\begin{center}
\begin{tabular}{|c|cc|}\hline
 & \textbf{$1$-Orbit Problem} & \textbf{$5$-Orbit Problem}  \\[2mm] \hline
\textbf{Utility-aware Privacy} & $3.5$ & $235.93$ \\[2mm]
\textbf{Utility-Only} & $0.5$ & $28.26$\\ \hline
\end{tabular}
\end{center}
\caption{Average computational times (sec) for $1$ and $5$ orbit problems. These times are obtained in MATLAB, with YALMIP \cite{lofberg2004yalmip} as the parser and MOSEK \cite{andersen2000mosek} as the solver.}
\label{tab:cpu}
\end{table}

Therefore, the utility-only optimization results in significantly faster solution, but the privacy is not maximized. Whereas, with the utility-aware privacy maximization, we are able to achieve higher privacy but with significantly higher computational time. Since both the formulations solve a semi-definite program, the computational times are expected to grow polynomially using interior-point method. Therefore, for large-scale problems, specialized methods (such as first order methods) can be applied to solve the optimization problem.

\section{Conclusion \& Summary}
In this paper we presented a new formulation that addresses optimal privacy vs utility tradeoff in the Ensemble/Unscented Kalman filtering framework. Privacy and utility are defined in terms of the estimation error variances. The formulation achieves this tradeoff by injecting synthetic noise to the sensor data, which is used to regulate the posterior error covariance. 

We demonstrated how this formulation can be applied to achieve the tradeoff in the context of space-object tracking problem. We were able to show that it is possible to satisfy utility (defined by upper bounding the estimation error), while achieving a certain level of privacy (defined by lower-bounding the estimation error).

In particular, the presented results demonstrated that:
\begin{enumerate}
\item Utility upper-bounds can be satisfied with sparse sensing, which also indirectly helps in sensor scheduling. This is achieved by satisfying the utility constraints with maximum sensor noise or minimum sensor precision, where the precision is defined to be the inverse of noise. Sparseness in sensing is achieved by $l_1$ minimization.
\item Privacy lower-bounds can be satisfied by maximizing the noise in the sensor.
\item A joint optimization problem for utility-aware privacy is able to maximize privacy (i.e. maximize the lower bound), while satisfying utility upper bound. Similarly, a joint-optimization can maximize utility (i.e. minimize the upper bound), while satisfying the privacy lower bound.
\end{enumerate}

These results have significant implications from a data sharing perspective. The above optimization problems determine the level of synthetic noise that should be added to the raw sensed data, such that the desired privacy-utility tradeoff is achieved. This is important, because privacy concerns can result in conservative data obfuscation and severely impede utility. On the other hand, utility concerns can result in excessive sensing accuracy, which can violate privacy concerns and perhaps be uneconomical from a system design and operations perspective. Therefore, we expect our framework to enable better data collection and sharing policy, particularly for the SSA community. In addition, since SSA data is being commoditized, the proposed algorithms will enable data-products with different levels of accuracy, corresponding to various stratified pricing models for future value-added services in space-data analytics.

\section{Acknowledgements}
This research was sponsored by AFOSR DDDAS grant FA9550-15-1-0071 and Air Force STTR Phase-I  grant FA8750-18-C-0106 (subcontracted by Intelligent Fusion Technology, Inc.).
\bibliographystyle{unsrt}
\bibliography{paper}

\begin{thebibliography}{10}

\bibitem{williamsen2008satellite}
Joel Williamsen.
\newblock {Satellite Vulnerability to Direct Ascent KE ASAT}.
\newblock {\em Survivability €"Time to Get Serious}, page~23, 2008.

\bibitem{foust2018spacex}
Jeff Foust.
\newblock Spacex's space-internet woes: Despite technical glitches, the company
  plans to launch the first of nearly 12,000 satellites in 2019.
\newblock {\em IEEE Spectrum}, 56(1):50--51, 2018.

\bibitem{geo1}
Veronica Magan.
\newblock {GOES 13 Weather Satellite Returns to Service After Hit from
  Micrometeoroid}.
\newblock \url{https://tinyurl.com/y5su7zs5}.

\bibitem{geo2}
{Meteosat-8 [MSG-1]}.
\newblock \url{http://sat-nd.com/failures/msg1.html}.

\bibitem{geo3}
{Russian Telecom Satellite Fails After Sudden Impact}.
\newblock \url{https://tinyurl.com/qcpd8}.

\bibitem{johnson2008history}
Nicholas~L Johnson, Eugene Stansbery, David~O Whitlock, Kira~J Abercromby, and
  Debra Shoots.
\newblock History of on-orbit satellite fragmentations.
\newblock {\em NASA/TM 2008 214779}, 2008.

\bibitem{mcknight2017preliminary}
Darren McKnight, Mark Matney, Kris Walbert, Sophie Behrend, Patrick Casey, and
  Seth Speaks.
\newblock Preliminary analysis of two years of the massive collision monitoring
  activity.
\newblock {\em NASA Technical Report}, 2017.

\bibitem{oltrogge2019technical}
Daniel~L Oltrogge and Salvatore Alfano.
\newblock The technical challenges of better space situational awareness and
  space traffic management.
\newblock {\em Journal of Space Safety Engineering}, 2019.

\bibitem{foust_2019}
Jeff Foust.
\newblock Esa spacecraft dodges potential collision with starlink satellite.
\newblock {\em SPACENEWS}, Sep 2019.

\bibitem{letgo}
Brian Weeden.
\newblock {Time for the U.S. military to let go of the civil space situational
  awareness mission}.
\newblock \url{https://tinyurl.com/y4886mqf}.

\bibitem{lal2018global}
Bhavya Lal, Asha Balakrishnan, Becaja~M Caldwell, Reina~S Buenconsejo, and
  Sara~A Carioscia.
\newblock Global trends in space situational awareness (ssa) and space traffic
  management (stm).
\newblock {\em IDA Science \& Technology Policy Institute}, 2018.

\bibitem{pelton2019space}
Joseph~N Pelton.
\newblock Space weapons, the threat of war in space and planetary defense.
\newblock In {\em Space 2.0}, pages 115--128. Springer, 2019.

\bibitem{hecht2019star}
Jeff Hecht.
\newblock A ``star wars'' sequel? the allure of directed energy for space
  weapons.
\newblock {\em Bulletin of the Atomic Scientists}, 75(4):162--170, 2019.

\bibitem{harrison2018space}
Todd Harrison, Kaitlyn Johnson, and Thomas~G Roberts.
\newblock {\em Space Threat Assessment 2018}.
\newblock Center for Strategic \& International Studies, 2018.

\bibitem{radtke2017interactions}
Jonas Radtke, Christopher Kebschull, and Enrico Stoll.
\newblock Interactions of the space debris environment with mega
  constellations?using the example of the oneweb constellation.
\newblock {\em Acta Astronautica}, 131:55--68, 2017.

\bibitem{space-directive3}
White House.
\newblock {Space Policy Directive-3, National Space Traffic Management Policy
  }.
\newblock
  \url{https://www.whitehouse.gov/presidential-actions/space-policy-directive-3-national-space-traffic-management-policy/}.

\bibitem{dwork2011differential}
Cynthia Dwork.
\newblock Differential privacy.
\newblock {\em Encyclopedia of Cryptography and Security}, pages 338--340,
  2011.

\bibitem{mcsherry2007mechanism}
Frank McSherry and Kunal Talwar.
\newblock Mechanism design via differential privacy.
\newblock In {\em FOCS}, volume~7, pages 94--103, 2007.

\bibitem{dwork2014algorithmic}
Cynthia Dwork, Aaron Roth, et~al.
\newblock The algorithmic foundations of differential privacy.
\newblock {\em Foundations and Trends{\textregistered} in Theoretical Computer
  Science}, 9(3--4):211--407, 2014.

\bibitem{cortes2016differential}
Jorge Cort{\'e}s, Geir~E Dullerud, Shuo Han, Jerome Le~Ny, Sayan Mitra, and
  George~J Pappas.
\newblock Differential privacy in control and network systems.
\newblock In {\em 2016 IEEE 55th Conference on Decision and Control (CDC)},
  pages 4252--4272. IEEE, 2016.

\bibitem{koufogiannis2017differential}
Fragkiskos Koufogiannis and George~J Pappas.
\newblock Differential privacy for dynamical sensitive data.
\newblock In {\em 2017 IEEE 56th Annual Conference on Decision and Control
  (CDC)}, pages 1118--1125. IEEE, 2017.

\bibitem{kawano2018differential}
Yu~Kawano and Ming Cao.
\newblock Differential privacy and qualitative privacy analysis for nonlinear
  dynamical systems.
\newblock {\em IFAC-PapersOnLine}, 51(23):52--57, 2018.

\bibitem{Farokhi_2019}
Farhad Farokhi and Henrik Sandberg.
\newblock Ensuring privacy with constrained additive noise by minimizing fisher
  information.
\newblock {\em Automatica}, 99:275--288, jan 2019.

\bibitem{song2017composition}
Shuang Song and Kamalika Chaudhuri.
\newblock Composition properties of inferential privacy for time-series data.
\newblock In {\em 2017 55th Annual Allerton Conference on Communication,
  Control, and Computing (Allerton)}, pages 814--821. IEEE, 2017.

\bibitem{ghosh2016inferential}
Arpita Ghosh and Robert Kleinberg.
\newblock Inferential privacy guarantees for differentially private mechanisms.
\newblock {\em arXiv preprint arXiv:1603.01508}, 2016.

\bibitem{sun2017inference}
Meng Sun and Wee~Peng Tay.
\newblock Inference and data privacy in iot networks.
\newblock In {\em 2017 IEEE 18th International Workshop on Signal Processing
  Advances in Wireless Communications (SPAWC)}, pages 1--5. IEEE, 2017.

\bibitem{evensen2003ensemble}
Geir Evensen.
\newblock The ensemble kalman filter: Theoretical formulation and practical
  implementation.
\newblock {\em Ocean dynamics}, 53(4):343--367, 2003.

\bibitem{julier1997new}
Simon~J Julier and Jeffrey~K Uhlmann.
\newblock New extension of the kalman filter to nonlinear systems.
\newblock In {\em Signal processing, sensor fusion, and target recognition VI},
  volume 3068, pages 182--193. International Society for Optics and Photonics,
  1997.

\bibitem{McCabe_2014}
James~S. McCabe and Kyle~J. DeMars.
\newblock Particle filter methods for space object tracking.
\newblock In {\em {AIAA}/{AAS} Astrodynamics Specialist Conference}. American
  Institute of Aeronautics and Astronautics, aug 2014.

\bibitem{vallado2001fundamentals}
David~A Vallado.
\newblock {\em Fundamentals of astrodynamics and applications}, volume~12.
\newblock Springer Science \& Business Media, 2001.

\bibitem{halder2011dispersion}
Abhishek Halder and Raktim Bhattacharya.
\newblock Dispersion analysis in hypersonic flight during planetary entry using
  stochastic liouville equation.
\newblock {\em Journal of Guidance, Control, and Dynamics}, 34(2):459--474,
  2011.

\bibitem{evensen1996assimilation}
Geir Evensen and Peter~Jan Van~Leeuwen.
\newblock Assimilation of geosat altimeter data for the agulhas current using
  the ensemble kalman filter with a quasigeostrophic model.
\newblock {\em Monthly Weather Review}, 124(1):85--96, 1996.

\bibitem{Song_2018}
Yang Song, Chong~Xiao Wang, and Wee~Peng Tay.
\newblock Privacy-aware kalman filtering.
\newblock In {\em 2018 {IEEE} International Conference on Acoustics, Speech and
  Signal Processing ({ICASSP})}. {IEEE}, apr 2018.

\bibitem{friersoneconomic}
Nate Frierson.
\newblock An economic goods analysis of us space situational awareness (ssa)
  policy.

\bibitem{anderson2012optimal}
Brian~DO Anderson and John~B Moore.
\newblock {\em Optimal filtering}.
\newblock Courier Corporation, 2012.

\bibitem{lofberg2004yalmip}
Johan L{\"o}fberg.
\newblock Yalmip: A toolbox for modeling and optimization in matlab.
\newblock In {\em Proceedings of the CACSD Conference}, volume~3. Taipei,
  Taiwan, 2004.

\bibitem{andersen2000mosek}
Erling~D Andersen and Knud~D Andersen.
\newblock The mosek interior point optimizer for linear programming: an
  implementation of the homogeneous algorithm.
\newblock In {\em High performance optimization}, pages 197--232. Springer,
  2000.

\end{thebibliography}

\end{document}